\documentclass[11pt,a4paper]{amsart}
\usepackage{amssymb,amsmath,epsfig,graphics,mathrsfs}

\usepackage{fancyhdr}
\usepackage{hyperref}
\hypersetup{
 colorlinks   = true,
 urlcolor     = blue,
 linkcolor    = blue,
 citecolor   = red ,
 bookmarksopen=true
}

\usepackage{amsmath}
\usepackage{amsfonts}
\usepackage{amssymb}
\usepackage{amsthm}
\usepackage{epsfig,graphics,mathrsfs}
\usepackage{graphicx}

\usepackage[usenames, dvipsnames]{color} 

\usepackage{hyperref}

\def \de {\partial}

\def \N {\mathbb{N}}

\def \phi {\varphi}
\def \RNu {\mathbb{R}^{N+1}}
\def \RN {\mathbb{R}^N}
\def \R {\mathbb{R}}

\def \K {\mathscr{K}}

\def \G{\Gamma}

\newcommand{\paa}{z^a \de_z}

\def \So {\mathscr{S}}
\newcommand{\As}{(-\mathscr A)^s}
\newcommand{\sA}{\mathscr A}

\newcommand{\Ia}{\mathscr I_\alpha}

\newcommand{\rpp}{\rho_p(\sA)}

\newcommand{\Rn}{\mathbb R^n}

\newcommand{\p}{\partial}

\newcommand{\la}{\lambda}

\numberwithin{equation}{section}

\newcommand{\beq}{\begin{equation}}
\newcommand{\bea}[1]{\begin{array}{#1} }
\newcommand{\eeq}{ \end{equation}}
\newcommand{\ea}{ \end{array}}

\newcommand{\ve}{\varepsilon}

\newcommand{\Rnp}{\mathbb R^{N+1}_+}

\newcommand{\Lo}{\mathscr L^{2s,p}}
\newcommand{\Ma}{\mathscr M}
\newcommand{\Po}{\mathscr P}

\newcommand{\Lp}{L^p}
\newcommand{\Li}{L^\infty}
\newcommand{\Lii}{L^\infty_0}

\newtheorem{theorem}{Theorem}[section]
\newtheorem{lemma}[theorem]{Lemma}
\newtheorem{proposition}[theorem]{Proposition}
\newtheorem{corollary}[theorem]{Corollary}
\newtheorem{remark}[theorem]{Remark}
\newtheorem{definition}[theorem]{Definition}

\numberwithin{equation}{section}

\begin{document}

\title[Hardy-Littlewood-Sobolev inequalities, etc.]{Hardy-Littlewood-Sobolev inequalities for a class of non-symmetric and non-doubling hypoelliptic semigroups}


\date{}

\begin{abstract}
In his seminal 1934 paper on Brownian motion and the theory of gases Kolmogorov introduced a second order evolution equation which displays some challenging features. In the opening of his 1967 hypoellipticity paper H\"ormander discussed a general class of degenerate Ornstein-Uhlenbeck operators that includes Kolmogorov's as a special case. In this note we combine semigroup theory with a nonlocal calculus for these hypoelliptic operators to establish new inequalities of Hardy-Littlewood-Sobolev type in the situation when the drift matrix has nonnegative trace. Our work has been influenced by  ideas of E. Stein  and Varopoulos in the framework of symmetric semigroups. One of our objectives is to show that such ideas can be pushed to successfully handle the present degenerate non-symmetric setting.  
\end{abstract}

\author{Nicola Garofalo}

\address{Dipartimento d'Ingegneria Civile e Ambientale (DICEA)\\ Universit\`a di Padova\\ Via Marzolo, 9 - 35131 Padova,  Italy}
\vskip 0.2in
\email{nicola.garofalo@unipd.it}

\thanks{The first author was supported in part by a Progetto SID (Investimento Strategico di Dipartimento) ``Non-local operators in geometry and in free boundary problems, and their connection with the applied sciences", University of Padova, 2017.}

\author{Giulio Tralli}
\address{Dipartimento d'Ingegneria Civile e Ambientale (DICEA)\\ Universit\`a di Padova\\ Via Marzolo, 9 - 35131 Padova,  Italy}
\vskip 0.2in
\email{giulio.tralli@unipd.it}

\maketitle

\tableofcontents

\section{Introduction}\label{S:intro}

Sobolev inequalities occupy a central position in analysis, geometry and physics. Typically, in such a priori estimates one is able to control a certain $L^q$ norm of a derivative of a function in terms of a $\Lp$ norm of derivatives of higher order. One distinctive aspect of these inequalities is that there is gain in the exponent of integrability, i.e., $q>p$. For instance, the prototypical Sobolev inequality in $\RN$ states that for any $1\le p<N$, there exists a constant $S_{N,p}$ such that for any function $f$ in the Schwartz class $\So$, one has
\[
(\star)\ \ \ ||f||_q \le S_{N,p}\ ||\nabla f||_p\ \ \ \Longleftrightarrow\ \ \ \ \frac 1p - \frac 1q = \frac 1N.
\]
In such framework, ($\star$) is referred to as the embedding theorem $W^{1,p}(\RN) \hookrightarrow L^q(\RN)$. The relation between the exponents $p$ and $q$ in ($\star$) is the well-known Hardy-Littlewood-Sobolev condition, and the if and only if character is connected with the interplay between the differential operator $\nabla$ and the homogeneous structure (Euclidean dilations) of the ambient space.   

In this paper we are concerned with a scale of global inequalities such as the one above for the following class of second-order partial differential equations in $\RNu$,
\begin{equation}\label{K0}
\mathscr K u = \mathscr A u  - \de_t u \overset{def}{=} \operatorname{tr}(Q \nabla^2 u) + <BX,\nabla u> - \de_t u = 0,
\end{equation}
where the $N\times N$ matrices $Q$ and $B$ have real, constant coefficients, and $Q = Q^\star \ge 0$. 
We assume throughout that $N\ge 2$, and we indicate with $X$ the generic point in $\R^N$, with $(X,t)$ the one in $\RNu$. It is worth noting here that when $Q = I_N$ and $B = O_N$, then \eqref{K0} becomes the standard heat operator $\Delta - \p_t$ in $\RNu$, and we are back into the framework of ($\star$). But in the degenerate case when $Q\ge 0$ and $B\not= O_N$, then the evolution of equations such as \eqref{K0} is driven by semigroups $P_t = e^{-t \sA}$ which, in general, are non-symmetric and non-doubling. Furthermore, there is no global homogeneous structure associated with them, and they lack an obvious notion of ``gradient". For instance, a tool like the P.A. Meyer \emph{carr\'e du champ} $\Gamma(f) = \frac 12 [\sA(f^2) - 2 f \sA f]$ is not directly effective here since $\G(f) = <Q\nabla f,\nabla f>$. This misses all directions of non-ellipticity in the degenerate case, and also does not provide control on the drift.

The class \eqref{K0} first appeared in the 1967 work of H\"ormander \cite{Ho}, in which he proved his celebrated hypoellipticity theorem asserting that if smooth vector fields $Y_0, Y_1,...,Y_m$ in $\RNu$ verify the finite rank condition on the Lie algebra, then the operator $\sum_{i=1}^m Y_i^2 + Y_0$ is hypoelliptic. To motivate this result, in the opening of his paper he discussed \eqref{K0} and showed that $\K$ is hypoelliptic if and only if
$\operatorname{Ker} Q$\ does not contain any non-trivial subspace which is invariant for\ $B^\star$. This condition can be equivalently expressed in terms of the strict positivity, hence invertibility, of 
the covariance matrix 
\begin{equation}\label{Kt}
K(t) = \frac 1t \int_0^t e^{sB} Q e^{s B^\star} ds
\end{equation}
for every $t>0$.
We note that, in the degenerate case when $Q$ fails to be elliptic, this property becomes void at $t = 0$, since $K(0) = Q$. Also, it is easy to see that $K(t)>0$ for every $t>0$ if and only if $K(t_0)>0$ for one $t_0>0$. Under the hypoellipticity assumption H\"ormander constructed a fundamental solution $p(X,Y,t)>0$ for \eqref{K0}, and proved that, given $f\in \So$, the Cauchy problem $\K u = 0, u(X,0) = f(X)$ admits a unique solution given by $P_t f(X) = \int_{\RN} p(X,Y,t) f(Y) dY$. This defines a non-symmetric semigroup $\{P_t\}_{t>0}$ which is strongly continuous in $\Lp$, $1\le p<\infty$, satisfies $P_t 1 = 1$, but which, however, is not contractive in general. 

Our primary interest in this paper is on the subclass of \eqref{K0} which, besides H\"ormander's hypoellipticity condition $K(t)>0$, also satisfy the assumption
\begin{equation}\label{trace}
\operatorname{tr} B \ge 0.
\end{equation} 
This serves to guarantee that the semigroup $\{P_t\}_{t>0}$ be contractive in $\Lp$ for $1\le p<\infty$, a fact that plays a pervasive role in our work. A prototypical example to keep in mind is the operator 
\[
\K_0 u = \Delta_v u + <v,\nabla_x u> - \p_t u,
\]
introduced by Kolmogorov in his seminal 1934 note \cite{Kol} on Brownian motion and the theory of gases. Here, we have let $N = 2n$, and $X = (v,x)$, with $v, x\in \Rn$. Such $\K_0$ fails to be parabolic since it is missing the diffusive term $\Delta_x u$, but it is easily seen to satisfy H\"ormander's finite rank condition for the hypoellipticity.  Equivalently, one can verify that $K(t)=\begin{pmatrix} I_n & t/2\ I_n\\ t/2\ I_n& t^2/3\ I_n\end{pmatrix}>0$ for every $t>0$. Remarkably, Kolmogorov himself had already produced the following explicit fundamental solution
\begin{align*}
p_0(X,Y,t) & = \frac{c_n}{t^{2n}} \exp\big\{- \frac 1t \big(|v-w|^2 
 + \frac 3t <v-w,y-x-tv>  + \frac{3}{t^2} |x- y +tv|^2\big)\big\},
\end{align*}
where $Y = (w,y)$.
Since such function is smooth off the diagonal, it follows that he had proved that $\K_0$ is hypoelliptic more than thirty years before \cite{Ho}.  We note that the hypothesis \eqref{trace} trivially includes Kolmogorov's operator $\K_0$ since for the latter we have $\operatorname{tr} B = 0$, but also encompasses several different examples of interest in mathematics and physics. For a short list  the reader can see the items in red in the table in fig.\ref{fig} in Section \ref{S:volume}. For the items in black (see \cite{OU}, \cite{UW}, \cite{Bri} and \cite{Fre}) we have $\operatorname{tr} B < 0$, thus they are not covered by our results. Such subclass of \eqref{K0} will be analysed in a future study. 

To provide the reader with some perspective we mention that during the last three decades there has been considerable progress in the study of the equations \eqref{K0}. The existing approaches are essentially of two types: a) far reaching adaptations of direct methods from partial differential equations combined with Lie group theory and analysis in spaces with homogeneous and non-homogeneous structures; or b) powerful combinations of ideas from probability and semigroup theory. For the existing literature covering either a) or b), the reader should see \cite{K82}, \cite{GL}, \cite{LP}, \cite{Pcat94}, \cite{L}, \cite{PARMA}, \cite{PoRa}, \cite{LaMo}, \cite{Me}, \cite{MPP}, \cite{Bou}, \cite{PP04}, \cite{PrZa}, \cite{FL}, \cite{DP}, \cite{CPP}, \cite{BCLP}, \cite{FNPP}, \cite{BRS}, \cite{NP}, \cite{Ny}, \cite{AT}, \cite{GIMV}, \cite{AM}, but such list is by no means exhaustive. One should also consult the survey papers \cite{LPP} and \cite{Bog}, and the books \cite{DZ92}, \cite{DZ}, \cite{HN}, \cite{LB} and \cite{Vi}. Despite such large body of works, some aspects presently remain elusive, such as: (i) a systematic development of an intrinsic Hardy-Littlewood-Sobolev theory;  (ii) the analysis of local and nonlocal isoperimetric inequalities.
The aim of the present paper is to take a first step in the program (i). In the work \cite{GTiso} we address (ii). 
 
Our approach combines semigroup theory with the nonlocal calculus for  \eqref{K0} recently developed in \cite{GT}, and it has been influenced by the ideas of E. Stein in \cite{Steinlp} and Varopoulos in \cite{V85} in the setting of positive symmetric semigroups. In fact, one of the objectives of the present paper is to show that their powerful ideas can be pushed to successfully handle the degenerate non-symmetric setting of \eqref{K0}. 
   
A discussion of the main results and techniques seems in order at this point. Section \ref{S:ks} is devoted to collecting the known background results on the semigroup $\{P_t\}_{t>0}$. We introduce the intertwined non-symmetric pseudo-distance $m_t(X,Y)$, and the time-dependent pseudo-balls $B_t(X,r)$. The volume function $V(t) = \operatorname{Vol}_N(B_t(X,\sqrt t))$ is defined in \eqref{VS}. The relevance of such function is demonstrated by its place in H\"ormander's probability transition density \eqref{PtKt}. We also recall for completeness an important result from \cite{LP} stating that as $t\to 0^+$ the small-time behaviour of $V(t)$ is governed by a suitable infinitesimal homogeneous structure. Using such information one can show that there exists $D_0\ge N\ge 2$ such that 
$V(t)\ \cong\ t^{D_0/2}$ as $t\to 0^+$.
We call the number $D_0$ the intrinsic dimension of the semigroup at zero. 

As it became evident from the work \cite{V85} (see also \cite{V80}, \cite{Viso} and \cite{VSC}), in Varopoulos' semigroup approach to the Hardy-Littlewood theory the evolution is driven by the large time behaviour of the semigroup. It should thus come as no surprise that the functional inequalities in this paper hinge on the behaviour of the volume function $V(t)$ as $t\to \infty$. Section \ref{S:volume} is dedicated to the analysis of this aspect. The first key result is Proposition \ref{P:boom} in which we show that, under the hypothesis \eqref{trace}, the function $V(t)$ must blow-up at least linearly as $t\to \infty$ (note that for the Ornstein-Uhlenbeck operator $\Delta_X - <X,\nabla_X> - \p_t$, for which $\operatorname{tr}(B) < 0$, one has instead $V(t)\to c_N>0$ as $t\to \infty$). Furthermore, if the drift matrix $B$ has at least one eigenvalue with strictly positive real part, then $V(t)$ blows up exponentially and is not doubling. In other words, in such situation the drift induces a negative ``curvature"  in the ambient space $\RN$. In Definition \ref{D:hld} we introduce the key notion of \emph{intrinsic dimension at infinity} of the semigroup, and we indicate such number with $D_\infty$. We note that the above mentioned minimal linear growth of $V(t)$ at infinity, provides the basic information that $D_\infty \ge 2$. The reader should see the table in fig.\ref{fig} where the quantities $D_0$ and $D_\infty$ are compared for several differential operators of interest in mathematics and physics.  The second result of the section is Proposition \ref{P:Koneinfty} which establishes the $\Lp-L^\infty$ ultracontractivity of the semigroup $\{P_t\}_{t>0}$ for $1\le p<\infty$. As the reader can surmise from the seminal work \cite[Theorem 1]{V85} in the symmetric case, such property plays a central role in our work as well.

In Section \ref{S:fpA} we introduce the relevant Sobolev spaces. One of the difficulties in the analysis of \eqref{K0}, already hinted at above, is that a ``gradient" is not readily available. This problem is circumvented using the nonlocal operator $(-\sA)^{1/2}$ as a gradient since it intrinsically contains the appropriate fractional order of differentiation along the drift, which is instead missing in the above mentioned carr\'e du champ. 
By means of Balakrishnan's formula \eqref{As}, we can precisely identify the nonlocal operators $\As$ by means of the semigroup $\{P_t\}_{t>0}$. This allows to introduce spaces of Sobolev type as follows. Given $0<s<1$ and $1\le p<\infty$, we define the Banach space
$$\Lo = \overline{\So}^{|| \  ||_{\Lo}},$$
where for a function in Schwartz class $\So$ we have denoted by  
$||f||_{\Lo} \overset{def}{=} ||f||_{L^p} + ||(-\sA)^s f||_{L^p}$.  We stress that, when $\sA = \Delta$, $s = 1/2$ and $1<p<\infty$, the classical Calder\'on-Zygmund theory guarantees that the space $\mathscr L^{1,p}$ coincides with the standard Sobolev space $
W^{1,p} = \{f\in L^p\mid \nabla f\in L^p\}$.

In Section \ref{S:poisson}, under the hypothesis \eqref{trace}, we establish a Littlewood-Paley estimate that has been so far missing in the analysis of the class  \eqref{K0}. To achieve this we have combined a far reaching idea of E. Stein in \cite{Steinlp} with the kernel associated with the Poisson semigroup $\Po_z = e^{z(-\sA)^{1/2}}$ in \cite{GT}. Combining such tools with the powerful abstract Hopf-Dunford-Schwartz ergodic theorem in \cite{DS} we obtain the main weak$-L^1$ estimate in Theorem \ref{T:maximal}.  
 
In Section \ref{S:fi} we introduce, for any $0<\alpha<D_\infty$, the Riesz potential operators $\Ia$.  Our central result is Theorem \ref{T:inverse} that shows that for any $0<\alpha<2$ and $f\in \So$, one has
\begin{equation}\label{azz}
f = \mathscr I_{\alpha} \circ (-\sA)^{\alpha/2} f = (-\sA)^{\alpha/2} \circ \Ia f.
\end{equation}
This proves that $\Ia = (-\sA)^{-\alpha/2}$. 
 Again, the hypothesis \eqref{trace} is essential. The reader should pay attention here to the fact noted above that, under such assumption, we have $D_\infty\ge 2$, and thus \eqref{azz} covers the whole range $0<\alpha<2$. We note that, once again, the semigroup $\Po_z = e^{z(-\sA)^{1/2}}$, $z>0$, is in the background here. 

In Section \ref{S:sobolev} we establish our main Hardy-Littlewood-Sobolev embedding, Theorem \ref{T:main}. Suppose that there exist $D, \gamma_D>0$ such that 
\begin{equation}\label{vol0}
V(t) \ge \gamma_D\ t^{D/2},\ \ \ \ \ \ \ \ \ \forall t>0.
\end{equation}
Then, for every $0<\alpha<D$ the operator $\Ia$ maps $L^1$ into $L^{\frac{D}{D-\alpha},\infty}$. If instead $1<p<D/\alpha$, then $\Ia$ maps $L^p$ to $L^q$, with $\frac 1p - \frac 1q = \frac{\alpha}D$. Combining this result with \eqref{azz} we finally obtain the Sobolev embedding Theorem \ref{T:sob}. We mention that in the ``negative curvature" situation when $D_\infty = \infty$, see in this respect the operator of Kolmogorov with friction in ex.$6^+$ in fig.\ref{fig}, given any $1\le p<\infty$ we are free to chose $D>\max\{D_0,2sp\}$ such that \eqref{vol0} hold. For such $D$ we thus obtain $\Lo \hookrightarrow L^{pD/(D-2sp)}$. The reader should note that \eqref{vol0} implies that $2\le D_0 \le D\le D_\infty$, and thus Theorems \ref{T:main} and \ref{T:sob} do not cover the possibility $D_0>D_\infty$. In the degenerate setting this case can occur, see the Ex. 4 of the Kramers' operator in fig.\ref{fig}. When $D_0>D_\infty$ the estimate \eqref{vol0} must be replaced by \eqref{vol2} below and, under such hypothesis, we obtain appropriate versions of the above described results, see Theorems \ref{T:mainsum} and \ref{T:sobsum}.

In closing, we compare our results with the available literature.
Presently, there exist very few Sobolev-type estimates related to the class of degenerate operators \eqref{K0}. In \cite{PP04, CPP} the authors prove some interesting local results for nonnegative solutions to equations modelled on \eqref{K0}. They use tools from potential theory and representation formulas. The restriction to solutions, however, does not allow to obtain a priori information for arbitrary functions. For kinetic Fokker-Planck equations (where in particular we have $X= (v,x)$, with $v$ indicating velocity and $x$ position), we mention the recent papers \cite{GIMV} and \cite{AM}. In the former the authors prove a local gain of integrability for nonnegative sub-solutions via a non-trivial adaptation of the so-called velocity averaging method. In the latter the authors obtain a Poincar\'e inequality in a weighted $L^2$ space by means of a ad-hoc variational space. Our results differ from either one of these works since our Sobolev spaces $\Lo$ are defined with the aid of the nonlocal operators $\As$. Similarly to the classical potential estimate $|f(X)|\le c_N I_1(|\nabla f|)(X)$, our formula \eqref{azz}, combined with Theorem \ref{T:sob}, provides the sharp a priori control of the $L^q$ norm of a function, in terms of the appropriate fractional order of differentiation. Both, along the directions of ellipticity, and of the drift.

We also mention \cite{Bou}, in which the author obtained $L^2$ a priori estimates for the above discussed homogeneous Kolmogorov's operator $\K_0$, and the work \cite{BCLP}, where the authors prove some Calder\'on-Zygmund type estimates (both in $L^p$ and weak-$L^1$) for the operator $\sA$. 
The interesting analysis in \cite{BCLP} combines local singular integral estimates with suitable coverings that exploit the homogeneous structure discovered in \cite{LP} (see also subsection \ref{SS:D0} below). Our approach, based on the semigroup $\Po_z = e^{z\sqrt{-\sA}}$, is different and allows to obtain results of a global nature, both in space and time.

\subsection{Notation} 

The notation $\operatorname{tr} A$ indicates the trace of a matrix $A$, $A^\star$ is the transpose of $A$, and $\nabla^2 u$ denotes the Hessian matrix of a function $u$.  All the function spaces in this paper are based on $\RN$, thus we will routinely avoid reference to the ambient space throughout this work. For instance, the Schwartz space of rapidly decreasing functions in $\RN$ will be denoted by $\So$, and for $1\le p \le \infty$ we let $\Lp = L^p(\RN)$. The norm in $\Lp$ will be denoted by $||\cdot||_p$, instead of $||\cdot||_{\Lp}$. We will indicate with $\Lii$ the Banach space of the $f\in C(\RN)$ such that $\underset{|X|\to \infty}{\lim}\ |f(X)| = 0$ with the norm $||\cdot||_\infty$. The reader should keep in mind the following simple facts: (1) $P_t : L^\infty_0 \to L^\infty_0$ for every $t>0$; (2) $\So$ is dense in $\Lii$.
 The notation  $|E|$ will indicate the $N$-dimensional Lebesgue measure of a set $E$. If $T:\Lp\to L^q$ is a bounded linear map, we will indicate with $||T||_{p\to q}$ its operator norm. If $ q =p$, the spectrum of $T$ on $\Lp$ will be denoted by $\sigma_p(T)$, the resolvent set by $\rho_p(T)$, the resolvent operator by $R(\la,T) = (\la I - T)^{-1}$. The notation $\operatorname{tr} A$ indicates the trace of a matrix $A$, $A^\star$ is the transpose of $A$, and $\nabla^2 u$ denotes the Hessian matrix of a function $u$. For $x>0$ we will indicate with $\G(x) = \int_0^\infty t^{x-1} e^{-t} dt$ Euler's gamma function.
For any $N\in \mathbb N$ we will use the standard notation
$\sigma_{N-1} = \frac{2\pi^{N/2}}{\G(N/2)}$, $\omega_N = \frac{\sigma_{N-1}}{N}$,
respectively for the $(N-1)$-dimensional measure of the unit sphere $\mathbb S^{N-1} \subset \RN$, and $N$-dimensional measure of the unit ball. We adopt the convention that $a/\infty = 0$ for any $a\in \R$.


\section{Preliminaries}\label{S:ks}

In this section we collect, mostly without proofs, various properties of the semigroup associated with \eqref{K0} which will be used throughout the rest of the paper. One should see \cite[Section 2]{GT}, where some of the results in this section are discussed in detail.

\subsection{One-parameter intertwined pseudo-distances}\label{SS:m}

Given matrices $Q$ and $B$ as in \eqref{K0} we introduce a one-parameter family of intertwined pseudo-distances which plays a key role in the analysis of the relevant operators $\K$. For $X, Y\in \RN$ we define
\begin{align}\label{m}
m_t(X,Y) & = \sqrt{<K(t)^{-1}(Y-e^{tB} X ),Y-e^{tB} X >},\ \ \ \ \ \ \ t>0.
\end{align}
It is obvious that, when $B\not= O_N$, we have $m_t(X,Y) \not= m_t(Y,X)$ for every $t>0$. Given $X\in \RN$ and $r>0$, we consider the set 
$$B_t(X,r) = \{Y\in \RN\mid m_t(X,Y) < r\},$$
and call it the time-varying pseudo-ball.
We will need the following simple result.

\begin{lemma}\label{L:slices}
Let $g\in L^1(0,\infty)$. Then, for every $X\in \RN$ and $t>0$ one has
\begin{equation}\label{g}
\int_{\RN} g(m_t(X,Y)) dY =  \sigma_{N-1} (\det K(t))^{1/2}\int_0^\infty g(r) r^{N-1} dr,
\end{equation}
and
\begin{equation}\label{gstar}
\int_{\RN} g(m_t(X,Y)) dX =  \sigma_{N-1} e^{-t \operatorname{tr} B} (\det K(t))^{1/2}\int_0^\infty g(r) r^{N-1} dr.
\end{equation}
In particular, we have
\begin{equation}\label{misB}
\operatorname{Vol}_N(B_t(X,r)) = \omega_N r^N (\det K(t))^{1/2}. \end{equation}

\end{lemma}

\begin{proof}
Formula \eqref{g} easily follows from \eqref{m} by the change of variable $Z = K(t)^{-1/2}(Y-e^{tB} X )$. The latter gives
\begin{align*}
\int_{\RN} g(m_t(X,Y)) dY = (\det K(t))^{1/2} \int_{\RN} g(|Z|) dZ = \sigma_{N-1} (\det K(t))^{1/2}\int_0^\infty g(r) r^{N-1} dr.
\end{align*}
The proof of \eqref{gstar} is similar and we leave it to the reader.
To obtain \eqref{misB} it suffices to apply \eqref{g} with $g = \mathbf 1_{(0,r)}$.

\end{proof}

We stress that the quantity in the right-hand side of \eqref{misB} is independent of $X\in \RN$, a reflection of the underlying Lie group structure induced  by the matrix $B$, see Remark \ref{R:lie}. As a consequence, we will hereafter drop the dependence in such variable and indicate 
$\operatorname{Vol}_N(B_t(X,r)) = V_t(r)$.
When $r = \sqrt t$, we will simply write $V(t)$, instead of $V_t(\sqrt t)$, i.e., 
\begin{equation}\label{VS}
V(t) = \operatorname{Vol}_N(B_t(X,\sqrt t)) = \omega_N  (\det(t K(t)))^{1/2}.
\end{equation}
 
\subsection{The Cauchy problem}\label{SS:hs}
We next recall the theorem in the opening of \cite{Ho} which constitutes the starting point of the present work. We warn the unfamiliar reader that our presentation of the fundamental solution \eqref{PtKt} of \eqref{K0} differs from that in \cite{Ho}. This is done to emphasise the role of the one-parameter intertwined pseudo-distances \eqref{m} and of the corresponding volume function $V(t)$ defined by \eqref{VS}. In \eqref{PtKt} below we have let $c_N = (4\pi)^{-N/2} \omega_N$.

\begin{theorem}[H\"ormander]\label{T:hor}
Given $Q$ and $B$ as in \eqref{K0}, for every $t> 0$ consider the \emph{covariance matrix} \eqref{Kt}.
Then, the operator $\K$ is hypoelliptic if and only if 
$\operatorname{det} K(t) >0$ for every $t>0$.
In such case, given $f \in \So$, the unique solution to the Cauchy problem $\K u = 0$ in $\R^{N+1}_+$, $u(X,0) = f$, is given by $P_t f(X) = \int_{\R^N} p(X,Y,t) f(Y) dY$,
where
\begin{equation}\label{PtKt}
p(X,Y,t) = \frac{c_N}{V(t)} \exp\left( - \frac{m_t(X,Y)^2}{4t}\right).
\end{equation}
\end{theorem}
For a small list of differential operators of interest in mathematics and physics that are encompassed by Theorem \ref{T:hor} the reader should see the table in fig.\ref{fig} at the end of this section. 

\begin{remark}\label{R:lie}
We mention that it was noted in \cite{LP} that the class \eqref{K0} is invariant with respect to the following non commutative group law $(X,s)\circ (Y,t) = (Y+ e^{-tB}X,s+t)$. Endowed with the latter, the space $(\R^{N+1},\circ)$ becomes a non-Abelian Lie group. This aspect is reflected in the expression \eqref{PtKt}, as well as in the invariance with respect to $\circ$ of the volume of the intertwined pseudoballs, see \eqref{misB} in Lemma \ref{L:slices}. Except for this, such Lie group structure will play no role in our work.
\end{remark}

\subsection{Semigroup aspects}\label{SS:sa}
In the following lemmas we collect the main (well-known) properties of the semigroup $\{P_t\}_{t>0}$ defined by Theorem  \ref{T:hor}. 

\begin{lemma}\label{L:invS}
For any $t>0$ we have: 
\begin{itemize}
\item[(a)] $\sA(\So)\subset \So$ and $P_t(\So) \subset \So$;
\item[(b)] For any $f\in \So$ and $X\in \RN$ one has $\frac{\de}{\de t} P_t f(X) = \mathscr A P_t f(X)$; 
\item[(c)] For every $f\in \So$, $X\in \RN$ and $t>0$, the commutation property is true
$\mathscr A P_t f(X) = P_t \mathscr A  f(X)$.
\end{itemize}
\end{lemma}

\begin{lemma}\label{L:Pt}
The following properties hold:
\begin{itemize}
\item[(i)] For every $X\in \RN$ and $t>0$ we have
$P_t 1(X) = \int_{\RN} p(X,Y,t) dY = 1$;
\item[(ii)] $P_t:L^\infty \to L^\infty$ with $||P_t||_{L^\infty\to L^\infty} \le 1$;
\item[(iii)] For every $Y\in \RN$ and $t>0$ one has
$\int_{\RN} p(X,Y,t) dX = e^{- t \operatorname{tr} B}$.
\item[(iv)] Let $1\le p<\infty$, then $P_t:L^p \to L^p$ with $||P_t||_{L^p\to L^p} \le e^{-\frac{t \operatorname{tr} B}p}$. If $\operatorname{tr} B\ge 0$, $P_t$ is a contraction on $L^p$ for every $t>0$;
\item[(v)] [Chapman-Kolmogorov equation]
for every $X, Y\in \R^N$ and $t>0$ one has
$p(X,Y,s+t)  = \int_{\R^N} p(X,Z,s) p(Z,Y,t) dZ$.
Equivalently, one has $P_{t+s} = P_t \circ P_s$ for every $s, t>0$.
\end{itemize}
\end{lemma}

\begin{lemma}\label{L:Lprate}
Let $1\le p \le \infty$. Given any $f\in \So$ for any $t\in [0,1]$ we have 
\[
||P_t f - f||_{p} \le ||\mathscr A f||_{p}\ \max\{1,e^{-\frac{\operatorname{tr} B}p}\}\ t.
\]
\end{lemma}

\begin{corollary}\label{C:Ptpzero}
Let $1\le p< \infty$. For every $f\in L^p$, we have
$||P_tf-f||_{p}\rightarrow 0$ as $t \to 0^+.$ Consequently, $\{P_t\}_{t>0}$ is a strongly continuous semigroup on $\Lp$. The same is true when $p = \infty$, if we replace $L^\infty$ by the space $\Lii$.  
\end{corollary}

\begin{remark}\label{R:infty}
The reader should keep in mind that from this point on when we consider $\{P_t\}_{t>0}$ as a strongly continuous semigroup in $\Lp$, we always intend to use $\Lii$ when $p = \infty$.
\end{remark}

Denote by $(\sA_p,D_p)$ the infinitesimal generator of the semigroup $\{P_t\}_{t>0}$ on $L^p$ with domain  
\[
D_p = \big\{f\in L^p\mid \sA_p f \overset{def}{=} \underset{t\to 0^+}{\lim}\ \frac{P_t f - f}{t}\ \text{exists in }\ L^p\big\}.
\]
One knows that $(\sA_p,D_p)$ is closed and densely defined (see \cite[Theorem 1.4]{EN}). 

\begin{corollary}\label{C:lp}
We have $\So\subset D_p$. Furthermore, $\sA_p f = \sA f$ for any $f\in \So$, and $\So$ is a core for $(\sA_p,D_p)$.
\end{corollary}

\begin{remark}\label{R:id}
From now on for a given $p\in [1,\infty]$ with a slight abuse of notation we write $\sA : D_p\to \Lp$ instead of $\sA_p$. In so doing, we must keep in mind that $\sA$ actually indicates the closed operator $\sA_p$ that, thanks to Corollary \ref{C:lp}, coincides with the differential operator $\sA$ on $\So$. Using this identification we will henceforth say that $(\sA,D_p)$ is the infinitesimal generator of the semigroup $\{P_t\}_{t>0}$ on $\Lp$.
\end{remark}

\begin{lemma}\label{L:specter}
Assume that \eqref{trace} be in force, and let $1\le p \le \infty$. Then: 
\begin{itemize}
\item[(1)] For any $\la\in \mathbb C$ such that $\Re \la >0$, we have $\la\in \rpp$;
\item[(2)] If $\la\in \mathbb C$ such that $\Re \la >0$, then $R(\la,\sA)$ exists and for any $f\in \Lp$ it is given by the formula $R(\la,\sA) f = \int_0^\infty e^{-\la t} P_t f\ dt$;
\item[(3)] For any $\Re \la > 0$ we have $||R(\la,\sA)||_{p\to p} \le \frac{1}{\Re \la}$.
\end{itemize}
\end{lemma}


\subsection{Small-time behaviour of the volume function}\label{SS:D0}
The small-time behaviour of the function $V(t)$ was studied in the paper \cite{LP}, where it was shown that the class of operators \eqref{K0} possesses an infinitesimal osculating structure. For completeness of presentation we recall it in this subsection. We begin with the following known result, see \cite{Ho}, \cite{LP}, \cite{L} and \cite{LB}. 

\begin{proposition}\label{P:equiv}
The following are equivalent:
\begin{itemize}
\item[(i)] condition $K(t)>0$ for every $t>0$;
\item[(ii)] $\operatorname{Ker} Q$\ does not contain any non-trivial subspace which is invariant for\ $B^\star$;
\item[(iii)] $\operatorname{Rank}[Q, B Q,...,B^{N-1} Q] = N$.   (Kalman rank condition)
\item[(iv)] the vector fields defined by $Y_0 u = <BX,\nabla u>$, $Y_i u = \sum_{i,j=1}^N a_{ij} \p_{X_j} u$, $i=1,...,N$,
where $A = [a_{ij}] = Q^{1/2}$, satisfy the finite rank condition
\[
\operatorname{Rank\ Lie}[Y_0,Y_1,...,Y_N](X) = N,\ \ \ \ \ \forall\ X\in \R^N.
\]
\item[(v)] in a suitable basis of $\RN$ the matrices $Q$ and $B$ assume the following form
\[
Q = \begin{pmatrix} Q_0 & 0 & \cdot & \cdot & \cdot & 0 & 0
\\
0 & 0 & \cdot & \cdot & \cdot & 0 & 0
\\
0 & 0 & \cdot & \cdot & \cdot &  0 & 0
\\
\cdot & \cdot & \cdot & \cdot & \cdot & \cdot & \cdot
\\
\cdot & \cdot & \cdot & \cdot & \cdot & \cdot & \cdot
\\
\cdot & \cdot & \cdot & \cdot & \cdot & \cdot & \cdot
\\
0 & 0 & \cdot & \cdot & \cdot &  0 & 0
\end{pmatrix},
\ \ \ \ \ \ \ \ 
B = \begin{pmatrix} \star & \star & \cdot & \cdot & \cdot & \star & \star
\\
B_1 & \star & \cdot & \cdot & \cdot & \star & \star
\\
0 & B_2 & \star & \cdot & \cdot &  \star & \star
\\
\cdot & \cdot & \cdot & \cdot & \cdot & \cdot & \cdot
\\
\cdot & \cdot & \cdot & \cdot & \cdot & \cdot & \cdot
\\
\cdot & \cdot & \cdot & \cdot & \cdot & \cdot & \cdot
\\
0 & 0 & \cdot & \cdot & \cdot &  B_r & \star
\end{pmatrix},
\]
where $Q_0$ is a $p_0\times p_0$ non-singular matrix, and $B_j$ is a $ p_j \times p_{j-1}$ matrix having rank $p_j$, $j=1,...,r$, with $p_0\ge p_1\ge ... \ge p_r \ge 1$, and $p_0 + p_1 + ... + p_r = N$. The $\star$ blocks in the canonical form of $B$ can be arbitrary matrices.
\end{itemize}
\end{proposition}

Let us now suppose that in a given basis of $\RN$ the matrices $Q$ and $B$ are given as in (v) of Proposition \ref{P:equiv}. Recall that $Q_0$ is a $p_0\times p_0$ positive matrix. We form a new matrix $\bar B$ by replacing all the elements with a $\star$ in $B$ with a zero matrix of the same dimensions, i.e.,  
\[
\bar B = \begin{pmatrix} 0 & 0 & \cdot & \cdot & \cdot & 0 & 0
\\
B_1 & 0 & \cdot & \cdot & \cdot & 0 & 0
\\
0 & B_2 & \cdot & \cdot & \cdot &  0 & 0
\\
\cdot & \cdot & \cdot & \cdot & \cdot & \cdot & \cdot
\\
\cdot & \cdot & \cdot & \cdot & \cdot & \cdot & \cdot
\\
\cdot & \cdot & \cdot & \cdot & \cdot & \cdot & \cdot
\\
0 & 0 & \cdot & \cdot & \cdot &  B_r & 0
\end{pmatrix}.
\]
We recall that $B_j$ is a $p_j\times p_{j-1}$ matrix having rank $p_j$. If we denote by $X=\left(x^{(p_0)},x^{(p_1)},\ldots,x^{(p_r)}\right)$ the generic point of $\RN=\R^{p_0}\times\R^{p_1}\times\cdots\times\R^{p_r}$, then the differential operator associated to the matrices $Q$ and $\bar B$ is given by
\begin{equation}\label{Kbar}
\bar{\K} = \operatorname{tr}\left(Q_0 D^2_{x^{(p_0)}}\right) + \sum_{j=1}^r <B_j x^{(p_{j-1})},\nabla_{x^{(p_j)}} > - \de_t.
\end{equation}
The fact that the blocks $B_j$ have maximal rank allows to easily check the condition (iv) in Proposition \ref{P:equiv}, therefore also $\bar \K$ verifies the H\"ormander's condition (i) in Proposition \ref{P:equiv}, with a matrix $\bar K(t)$ defined as in \eqref{Kt} with $\bar B$ in place of $B$. Furthermore, $\bar \K$ is left-invariant with respect to the group law $\circ$ in Remark \ref{R:lie}, in which $B$ has been replaced by $\bar B$. We remark that $\operatorname{tr} \bar  B = 0$, and that $\bar B$ is nilpotent, therefore $e^{s \bar B}$ is in fact a finite sum. One important aspect of the operator $\bar \K$  is that, unlike $\K$, it possesses a homogeneous structure: it is invariant of degree $2$ with respect to the group of anisotropic dilations $\delta_\lambda : \RNu\longrightarrow\RNu$ defined by
\begin{equation}\label{sdl}
\delta_\lambda(X,t) = \left(\lambda
x^{(p_0)},\lambda^3
x^{(p_1)},\ldots,\lambda^{2r+1}x^{(p_r)},\lambda^2t\right).
\end{equation}
We mention that it was proved in \cite[Proposition 2.2]{LP} that a necessary and sufficient condition for the existence of a family of non-isotropic dilations $\delta_\lambda$ associated with the operator $\K$ in \eqref{K0} is that $B$ in (v) takes precisely the special form $\bar B$.
 The \emph{homogeneous dimension} of $(\RNu, \circ, \delta_\lambda)$ is given by
\begin{equation}\label{hd}
D_0+2=p_0+3p_1+\ldots+(2r+1)p_r+2.
\end{equation}

Returning to the general discussion, we consider the one-parameter group of anisotropic dilations $D_\lambda:\RN\longrightarrow\RN$ associated with the matrix $\bar B$ 
\begin{equation}\label{Dl}
\quad D_\lambda(X) = \left(\lambda
x^{(p_0)},\lambda^3
x^{(p_1)},\ldots,\lambda^{2r+1}x^{(p_r)}\right).
\end{equation}
The fact that $\delta_\lambda$ are group automorphisms with respect to $\circ$ is a consequence of the following commutation property valid for any $\lambda>0$ and $\tau\in\R$, 
$$e^{-\lambda^2\tau \bar B}=D_\lambda e^{-\tau \bar B}D_{\lambda^{-1}},$$
(see \cite[eq. (2.20)]{LP} and also \cite{K82}).
From this, and the fact that $\operatorname{tr} \bar B = 0$, one can see that the positive definite matrix $\bar K(t)$,  defined in \eqref{Kt} with $\bar B$ instead of $B$, satisfies 
\[
\operatorname{det}(t \bar K(t)) = t^{D_0} \operatorname{det}(\bar K(1)).
\]
Denoting with $\bar V(t)$ the volume of the pseudoballs $\bar B_t(X,\sqrt t)$ associated with $\bar K$, we thus conclude that we must have for every $t>0$,
\begin{equation}\label{osc}
\bar V(t) = c_N \operatorname{det}(\bar K(1))^{1/2}\ t^{D_0/2} = \gamma\ t^{D_0/2}.
\end{equation}
The result in \cite[eq. (3.14) and Remark 3.1]{LP} gives us the following asymptotic.

\begin{proposition}\label{P:KbarK}
One has
$\underset{t\to 0^+}{\lim}\ \frac{V(t)}{\bar V(t)} = 1$.
\end{proposition}

Proposition \ref{P:KbarK} and \eqref{osc} motivate the following.

\begin{definition}\label{D:D0}
We call the number $D_0$ in \eqref{hd} the \emph{intrinsic dimension at zero} of the H\"ormander semigroup $\{P_t\}_{t>0}$. Note that it follows from \eqref{hd} that it must be $D_0\ge N\ge 2$.
\end{definition}

 
 \section{Large time behaviour of the volume function and ultracontractivity}\label{S:volume}

The analysis of the semigroup $\{P_t\}_{t>0}$ revolves on the large time behaviour of  the volume function $V(t)$. In this section we analyse this behaviour under the assumption \eqref{trace}. Our main result, Proposition \ref{P:boom}, plays a pervasive role in the rest of the paper since: 1) it shows that $V(t)$ grows at infinity at least linearly; and, 2) it says that when at least one of the eigenvalues of the drift matrix $B$ has a strictly positive real part, then $V(t)$ must blow up exponentially.
In what follows we will make use of the equivalence (i)\ $\Longleftrightarrow$\ (ii) in Proposition \ref{P:equiv}. The notation $\sigma(B)$ indicates the spectrum of $B$. 

\begin{proposition}\label{P:boom}
Suppose that \eqref{trace} hold. Then:
\begin{itemize}
\item[(i)] there exists a constant $c_1>0$ such that $V(t)\geq c_1 t$ for all $t\geq 1$;
\item[(ii)] moreover, if $\max\{\Re(\lambda)\mid \lambda\in \sigma(B)\}=L_0>0$, there exists a constant $c_0$ such that $V(t)\geq c_0 e^{L_0 t}$ for all $t\geq 1.$
\end{itemize}
\end{proposition}

\begin{proof}
As it will be evident from the proof, we first establish (ii) and then (i). Up to a change of variables in $\R^N$, we can assume that $B^*$ is in the following block-diagonal real Jordan canonical form (see, e.g., \cite[Theorem 3.4.5]{HJ})
$$B^\star=\begin{pmatrix}
J_{n_1}(\lambda_1) &  &  &  &  &  & 
\\
 & \ddots &  &  & 0 &   & 
\\
 &  & J_{n_q}(\lambda_q) &  &  &  & \\
 &  &  & C_{m_1}(a_1,b_1) &  &  & 
\\
 &  & 0 &  &  & \ddots & 
\\
 &  &  &  &  &   & C_{m_p}(a_p,b_p)
\end{pmatrix},$$
where $\sigma(B)=\sigma(B^\star)=\{\lambda_1,\ldots,\lambda_q,a_1\pm ib_1,\ldots, a_p\pm ib_p\}$ with $\lambda_k, a_\ell, b_\ell\in \R$ ($b_\ell\neq 0$), $n_1+\ldots+n_q+2m_1+\ldots+2m_p=N$ with $n_k, m_\ell\in \N$, and the $n_k\times n_k$ matrix $J_{n_k}(\lambda_k)$ and the $2m_\ell\times 2m_\ell$ matrix $C_{m_\ell}(a_\ell,b_\ell)$
are respectively in the form
$$J_{n_k}(\lambda_k)=\begin{pmatrix}
\lambda_k & 1 & 0 & \ldots & 0  \\
0 & \lambda_k & 1 & \ldots & 0  \\
0 & 0 & \ddots & \ddots & 0  \\
0 & \ldots & 0 & \lambda_k & 1  \\
0 & 0 & \ldots & 0 & \lambda_k 
\end{pmatrix},\quad C_{m_\ell}(a_\ell,b_\ell)=\begin{pmatrix}
a_\ell & -b_\ell & 1 & 0 & 0 & \ldots & \ldots & 0  \\
b_\ell & a_\ell & 0 & 1 & 0 & \ldots & \ldots & 0  \\
0 & 0 & a_\ell & -b_\ell & 1 & 0 & \ldots & 0  \\
0 & 0 & b_\ell & a_\ell & 0 & 1 & \ldots & 0 \\
\vdots & \vdots & 0 & 0 & \ddots & \ddots & 1 & 0 \\
\vdots & \vdots & 0 & 0 & \ddots & \ddots & 0 &  1\\
0 & 0 & \ldots & \ldots & 0 & 0 & a_\ell & -b_\ell \\
0 & 0 & \ldots & \ldots & 0 & 0 & b_\ell & a_\ell \\
\end{pmatrix}.
$$
Since $\operatorname{tr} B=\sum_{k=1}^q \lambda_k + 2\sum_{\ell=1}^p a_\ell \ge 0$, we have two cases:
$$\mbox{either}\qquad L_0=\max\{\lambda_k,a_\ell\}>0\qquad \mbox{or}\qquad \lambda_k=0=a_\ell\,\,\forall k\in\{1,\ldots,q\},\, \ell\in\{1,\ldots,p\}.$$
Suppose $L_0>0$. We are going to show that, for some $C_0>0$, we have
\begin{equation}\label{detL0}
\operatorname{det}\left(tK(t)\right)\geq C_0 e^{2L_0 t}\qquad\mbox{for all }t\geq 1.
\end{equation}
To do this, it is enough to show that
\begin{equation}\label{lmaxL0}
\lambda_M(t)\geq C e^{2L_0 t}\qquad\mbox{for all }t\geq 1,
\end{equation}
where $\lambda_M(t)$ is the largest eigenvalue of $tK(t)$. In fact, since $t\mapsto tK(t)$ is monotone increasing in the sense of matrices, for $t\geq 1$ all the eigenvalues of $tK(t)$ are larger than the minimum eigenvalue of $K(1)$, which is strictly positive by H\"ormander condition: this tells us that \eqref{lmaxL0} implies \eqref{detL0}. To prove \eqref{lmaxL0}, we notice that at least one of the following two possibilities occurs:
\begin{itemize}
\item[(a)] there is $k_0\in\{1,\ldots,q\}$ such that $\lambda_{k_0}=L_0$;
\item[(b)] there is $\ell_0\in\{1,\ldots,p\}$ such that $a_{\ell_0}=L_0$.
\end{itemize}
Suppose case $(a)$ occurs. It is not restrictive to assume $k_0=1$. Then, $v_0=(1,0,\ldots,0)\in\RN$ is an eigenvector for $B^*$ with relative eigenvalue $L_0$. Thus $e^{sB^\star}v_0=e^{L_0 s} v_0$, for all $s\in\R$. From (ii) in Proposition \ref{P:equiv} we know that $v_0\notin \operatorname{Ker} Q$, i.e. $\left\langle Q v_0,v_0\right\rangle>0$. Therefore, we have
$$\lambda_M(t)\geq \left\langle tK(t) v_0,v_0\right\rangle =\int_0^t \left\langle Q e^{sB^\star}v_0, e^{sB^\star}v_0\right\rangle\,ds = \left\langle Q v_0,v_0\right\rangle \int_0^t e^{2L_0 s}\,ds=\frac{e^{2L_0t}-1}{2L_0}\left\langle Q v_0,v_0\right\rangle,$$
which proves \eqref{lmaxL0} in case $(a)$. Suppose now case $(b)$ occurs. As before, let us assume also $\ell_0=1$ (that is $a_1=L_0$). If $e_1,e_2\in\R^{2m_1}$ are the vectors $e_1=(1,0,\ldots,0)$, $e_2=(0,1,0,\ldots,0)$, we denote the correspondent vectors in $\RN$ by $v_1=(0,\ldots,0,e_1,0,\ldots,0)$, $v_2=(0,\ldots,0,e_2,0,\ldots,0)$. With these notations, we have that $\operatorname{span}\{v_1,v_2\}$ is an invariant subspace for $B^\star$. From (ii) in Proposition \ref{P:equiv} we know that $\operatorname{span}\{v_1,v_2\}$ is not contained in $\operatorname{Ker} Q$. Moreover, denoting by $J$ the simplectic matrix restricted to $\operatorname{span}\{v_1,v_2\}$ such that $Jv_1=v_2$ and $Jv_2=-v_1$, we have 
\begin{equation}\label{expJ}
e^{sB^\star}v= e^{L_0 s}\left(\cos(b_1 s) v + \sin(b_1 s)Jv\right)\quad\mbox{ for all }s\in\R\mbox{ and for any }v\in \operatorname{span}\{v_1,v_2\}.
\end{equation}
Hence, for $v\in\operatorname{span}\{v_1,v_2\}$, we have
\begin{align*}
&\lambda_M(t)\geq \left\langle tK(t) v,v\right\rangle =\int_0^t \left\langle Q e^{sB^\star}v, e^{sB^\star}v\right\rangle\,ds \\
&= \int_0^t e^{2L_0 s}\left(\cos^2(b_1 s) \left\langle Qv,v\right\rangle + \sin^2(b_1 s) \left\langle QJv,Jv\right\rangle  + \sin(2b_1 s) \left\langle Qv,Jv\right\rangle\right)\,ds\\
&=\frac{e^{2L_0t}}{4L_0}\left(\left\langle Q\left(\cos(b_1t)v+\sin(b_1t)Jv\right),\left(\cos(b_1t)v+\sin(b_1t)Jv\right) \right\rangle + \frac{1}{L_0^2+b_1^2}\left\langle Q v_t,v_t\right\rangle \right)\\
&-\frac{1}{4L_0}\left(\left\langle Qv,v\right\rangle + \left\langle QJv,Jv\right\rangle\right)-\frac{L_0}{4(L_0^2+b_1^2)}\left(\left\langle Qv,v\right\rangle - \left\langle QJv,Jv\right\rangle\right)+\frac{b_1}{2(L_0^2+b_1^2)}\left\langle Qv,Jv\right\rangle,
\end{align*}
where we have denoted $v_t=(L_0\cos(b_1t)+b_1\sin(b_1t))v-(b_1\cos(b_1t)-L_0\sin(b_1t))Jv$. The fact that $v$ and $Jv$ cannot belong both to $\operatorname{Ker} Q$ implies that, for any $t>0$, also $v_t$ and $\cos(b_1t)v+\sin(b_1t)Jv$ cannot be in $\operatorname{Ker} Q$ at the same time. This says, since $Q\geq 0$, that
$$\left\langle Q\left(\cos(b_1t)v+\sin(b_1t)Jv\right),\left(\cos(b_1t)v+\sin(b_1t)Jv\right) \right\rangle + \frac{1}{L_0^2+b_1^2}\left\langle Q v_t,v_t\right\rangle \geq \bar{c}$$
for some positive $\bar{c}$, from which we can deduce
$$\lambda_M(t)\geq \frac{\bar{c}}{8L_0}e^{2L_0t}\quad\mbox{ for large t}.$$
This proves \eqref{lmaxL0} and concludes case $(b)$. This establishes (ii) in the statement of the proposition.\\
We next turn to proving (i). Suppose that
$$\lambda_k=0=a_\ell\,\,\forall k\in\{1,\ldots,q\},\, \ell\in\{1,\ldots,p\}.$$
Since $N\geq 2$, at least one of the following possibilities must occur:
\begin{itemize}
\item[(1)] there are two linearly independent eigenvectors of $B^*$ with eigenvalue $0$;
\item[(2)] there exists $k_0\in\{1,\ldots,p\}$ such that $n_{k_0}\geq 2$;
\item[(3)] there exists $\ell_0\in\{1,\ldots,p\}$ such that $m_{\ell_0}\geq 1$.
\end{itemize}
In each of these three cases we are going to show that, for some $C_1>0$, we have
\begin{equation}\label{dettutto0}
\operatorname{det}\left(tK(t)\right)\geq C_1 t^2\qquad\mbox{for all }t\geq 1.
\end{equation}
The case $(1)$ is the easiest since we have two linearly independent eigenvectors $v_{0,1}, v_{0,2}$ such that $e^{sB^\star}v_{0,j}=v_{0,j}$ for all $s\in\R$ and $j\in\{1,2\}$. Furthermore, for any $v\in \operatorname{span}\{v_{0,1},v_{0,2}\}$, $\left\langle Qv,v\right\rangle>0$ since the eigenvectors cannot belong to $\operatorname{Ker} Q$ by (ii). Hence,
$$\left\langle tK(t)v,v\right\rangle=\int_0^t \left\langle Q e^{sB^\star}v, e^{sB^\star}v\right\rangle\,ds = t \left\langle Qv,v\right\rangle$$
for all $t>0$ and $v\in \operatorname{span}\{v_{0,1},v_{0,2}\}$. Then, the symmetric matrix $tK(t)$ has at least two eigenvalues growing as $t$. By monotonicity the other eigenvalues are bounded below by a positive constant for $t\geq 1$. This yields \eqref{dettutto0}.\\
Suppose case $(2)$ occurs. Again, it is not restrictive to assume $k_0=1$. Then, we have $n_1\geq 2$ and we know that
$$e^{sJ_{n_1}(0)} =\begin{pmatrix}
1 & s & \frac{s^2}{2} & \ldots & \frac{s^{n_1-1}}{(n_1-1)!}  \\
0 & 1 & s & \ddots &  \frac{s^{n_1-2}}{(n_1-2)!} \\
0 & 0 & \ddots & \ddots & \ddots  \\
0 & \ldots & 0 & 1 & s  \\
0 & 0 & \ldots & 0 & 1 
\end{pmatrix}.$$
From this we infer that $e^{sB^\star}e_{n_1}=\frac{s^{n_1-1}}{(n_1-1)!} e_1 + o(s^{n_1-1})$ as $s\rightarrow+\infty$, where we have denoted by $e_j$ the canonical basis of $\RN$. We recall that, being $e_1$ an eigenvector for $B^\star$, we have $\left\langle Qe_1,e_1\right\rangle>0$. Therefore, we obtain
$$\left\langle Q e^{sB^\star}e_{n_1}, e^{sB^\star}e_{n_1}\right\rangle= \left\langle Q e_{1}, e_{1}\right\rangle \frac{s^{2n_1-2}}{(n_1-1)!^2}+ o(s^{2n_1-2})$$
as $s\rightarrow+\infty$. In particular, for $s\geq s_0$, we deduce
$$\left\langle Q e^{sB^\star}e_{n_1}, e^{sB^\star}e_{n_1}\right\rangle\geq c s^{2n_1-2},$$
for some positive constant $c$. This implies
$$\left\langle tK(t) e_{n_1},e_{n_1}\right\rangle\geq \int_{s_0}^t \left\langle Q e^{sB^\star}e_{n_1}, e^{sB^\star}e_{n_1}\right\rangle\,ds \geq c\int_{s_0}^t s^{2n_1-2}\,ds\geq \tilde{c}t^{2n_1-1} $$
for large $t$. This tells us that $\lambda_M(t)\geq \tilde{c}t^{2n_1-1}$, from which \eqref{dettutto0} follows since $n_1\geq 2$.\\
We are left with case $(3)$. As before, assume $\ell_0=1$ and denote by $v_1=(0,\ldots,0,e_1,0,\ldots,0)$, $v_2=(0,\ldots,0,e_2,0,\ldots,0)$ the vectors in $\RN$ which correspond to the canonical vectors $e_1, e_2\in\R^{2m_1}$. We have that $\operatorname{span}\{v_1,v_2\}$ is an invariant subspace for $B^\star$. From (ii) in Proposition \ref{P:equiv} we know that $\operatorname{span}\{v_1,v_2\}$ is not contained in $\operatorname{Ker} Q$. Thus, at least one of $\left\langle Qv_1,v_1\right\rangle$ and $\left\langle Qv_2,v_2\right\rangle$ is strictly positive.
With the same notations as in \eqref{expJ}, for $v\in\operatorname{span}\{v_1,v_2\}$ we find
$$e^{sB^\star}v= \cos(b_1 s) v + \sin(b_1 s)Jv.$$
Hence, for all $t>0$ we have
\begin{align*}
&\left\langle tK(t) v,v\right\rangle =\int_0^t \left\langle Q e^{sB^\star}v, e^{sB^\star}v\right\rangle\,ds \\
&=\left\langle Qv,v\right\rangle \int_0^t \cos^2(b_1 s)\,ds +\left\langle QJv,Jv\right\rangle \int_0^t \sin^2(b_1 s)\,ds +2\left\langle Qv,Jv\right\rangle \int_0^t \cos(b_1 s)\sin(b_1 s)\,ds \\
&=\left(\frac{t}{2}+\frac{\sin(2b_1t)}{4b_1}\right)\left\langle Qv,v\right\rangle+\left(\frac{t}{2}-\frac{\sin(2b_1t)}{4b_1}\right)\left\langle QJv,Jv\right\rangle +\frac{1-\cos(2b_1t)}{2b_1}\left\langle Qv,Jv\right\rangle. 
\end{align*}
Exploiting the fact that $\left\langle Qv,v\right\rangle + \left\langle QJv,Jv\right\rangle=\left\langle Qv_1,v_1\right\rangle + \left\langle Qv_2,v_2\right\rangle>0$, we can make the quantity $\left\langle tK(t) v,v\right\rangle\geq c t$ for large $t$ and for any $v\in\operatorname{span}\{v_1,v_2\}$. This shows the validity of \eqref{dettutto0}, thus completing the proof.

\end{proof}

Proposition \ref{P:boom} has the following basic consequence.
 
\begin{corollary}\label{C:infty}
Suppose that \eqref{trace} hold. 
Then, $\lim_{t\rightarrow \infty} V(t) = \infty.$
\end{corollary}

\begin{proof}
Recalling that $t\to t K(t)$ is monotone increasing in the sense of matrices, we have that $t\to V(t)$ is also a monotone function. Then, the conclusion immediately follows from Proposition \ref{P:boom}.
\end{proof}

\subsection{Intrinsic dimension at infinity}\label{SS:Dinfty}

In dealing with the general class \eqref{K0} the first question that comes to mind is: what number occupies the role of the dimension $N$ in the analysis of the semigroup $\{P_t\}_{t>0}$? This question is central since, as one can see in fig.\ref{fig}, the behaviour for large times of the volume function $V(t) = \operatorname{Vol}_N(B_t(X,\sqrt t))$ can be quite diverse, depending on the structure of the matrix $B$, and in fact non-doubling in general. The next definition introduces a notion which allows to successfully handle this matter. 

\begin{definition}\label{D:hld}
Consider the set $\Sigma_\infty \overset{def}{=} \big\{\alpha>0\big| \int_1^\infty \frac{t^{\alpha/2-1}}{V(t)} dt < \infty\big\}$.
We call the number  $D_\infty = \sup \Sigma_\infty$ the \emph{intrinsic dimension at infinity} of the semigroup $\{P_t\}_{t>0}$. 
\end{definition}

When $\Sigma_\infty = \varnothing$ we set $D_\infty = 0$. If $\Sigma_\infty \not= \varnothing$ we clearly have $0<D_\infty\le \infty$.

\begin{remark}\label{R:volumes}
Some comments are in order:
\begin{itemize}
\item[(1)] when $\sA = \Delta$, the standard Laplacian in $\RN$, then $V(t) = \omega_N t^{N/2}$. In such case, $\alpha\in \Sigma_\infty$ if and only if $0<\alpha<N$, and thus $D_\infty = N$;
\item[(2)] if the operator $\K$ in \eqref{K0} admits a homogeneous structure (as $\bar{\K}$ in subsection \ref{SS:D0}), we have $V(t) = \gamma_N t^{D_0/2}$, where $D_0$ is the intrinsic dimension at zero of the semigroup. In such case, we have $D_\infty = D_0$;
\item[(3)] more in general, if there exist constants $T, \gamma, D>0$ such that $V(t) \ge \gamma_D\ t^{D/2}$ for every $t\ge T$, then we must have $(0,D)\subset \Sigma_\infty$, and therefore $D_\infty\ge D$; from this observation and (i) in Proposition \ref{P:boom} we infer $D_\infty\ge 2$;
\item[(4)] the reader should note that the assumption $\Sigma_\infty \not= \varnothing$ hides a condition on the matrix $B$ in \eqref{K0}. For instance, when $Q = I_N$ and $B = - I_N$, then $\K u= \Delta u - <X,\nabla u> - \p_t u$ is the Ornstein-Uhlenbeck operator, see Ex.2 in fig.\ref{fig} below. In such case we have $V(t) = c_N (1-e^{-2t})^N$, and therefore in particular $\underset{t\to \infty}{\lim}\ V(t) = c_N>0$. It follows that $\Sigma_\infty = \varnothing$ and thus $D_\infty = 0$. The same happens with the Smoluchowski-Kramers' operator in Ex.5 below. In both cases the theory developed in this paper does not apply (we will return to this aspect in a future study);
\item[(5)] it can happen that $D_\infty < D_0$, see Ex.4;
\item[(6)] finally, one can have $D_\infty = \infty$, see Ex.$6^+$.
\end{itemize}
\end{remark}

In the following table we illustrate the different behaviours of the volume function $V(t)$ on a significant sample of operators. The items in red refer to situations in which the drift matrix satisfies $\operatorname{tr}(B)\ge 0$. This is the situation covered by this paper.

\vskip 0.2in 

\begin{figure}[htbp]
\begin{tabular}{| c | c | c | c | c | c | c | }
\hline \vphantom{$\left(\frac{t^2}{4}+\frac{1}{8}\right)^{\frac{1}{2}}$} Ex. &  {\color{red}$\K$} & {\color{red}$V(t)$} & {\color{red}$\operatorname{tr}(B)$}  & {\color{red}$N$}  &  {\color{red}$D_0$} & {\color{red}$D_{\infty}$} \\ 
\hline \vphantom{$\left(\frac{t^2}{4}+\frac{1}{8}\right)^{\frac{1}{2}}$}{\color{red}(1)} & $\underset{\text{Heat}}{\Delta - \p_t}$ &  $\omega_N t^{\frac{N}{2}}$   &  {\color{red}$0$}  & {\color{red}$N$} & {\color{red}$N$} &  {\color{red}$N$}        \\ 
\hline \vphantom{$\left(\frac{t^2}{4}+\frac{1}{8}\right)^{\frac{1}{2}}$}(2) &  $\underset{\text{Ornstein-Uhlenbeck}}{\Delta - <X,\nabla> - \p_t}$ &  $\omega_N 2^{-\frac{N}{2}} (1-e^{-2t})^{\frac{N}{2}}$   &  $-N$ & $N$ & $N$ & $0$         \\ 
\hline \vphantom{$\left(\frac{t^2}{4}+\frac{1}{8}\right)^{\frac{1}{2}}$}{\color{red}(3)} & $\underset{\text{Kolmogorov}}{\Delta_v  + <v,\nabla_x > - \p_t}$ &  $\omega_{2n} 12^{-\frac{n}{2}} t^{2n}$  & {\color{red}$0$}  & {\color{red}$2n$} & {\color{red}$4n$} & {\color{red}$4n$}         \\ 
\hline \vphantom{$\left(\frac{t^2}{4}+\frac{1}{8}\right)^{\frac{1}{2}}$}{\color{red}(4)} & $\underset{\text{Kramers}}{\p_{vv} - x \p_v + v \p_x - \p_t}$  &  $\pi\left(\frac{t^2}{4}+\frac{1}{8}\left(\cos(2t)-1\right)\right)^{\frac{1}{2}}$   &  {\color{red}$0$} & {\color{red}$2$} &  {\color{red}$4$}  &  {\color{red}$2$}   \\
\hline \vphantom{$\left(\frac{t^2}{4}+\frac{1}{8}\right)^{\frac{1}{2}}$}(5) & $\underset{\text{Smoluchowski-Kramers}}{\p_{vv}   - 2(v+x) \p_v  + v \p_x  - \p_t }$   &   $\frac{\pi}{4\sqrt{2}} \left(e^{-4t}+1-2e^{-2t}(2-\cos(2t))\right)^{\frac{1}{2}}$  & $-2$  &  $2$ &  $4$  & $0$  \\   
\hline \vphantom{$\left(\frac{t^2}{4}+\frac{1}{8}\right)^{\frac{1}{2}}$}${\color{red}(6^+)}$ & $\underset{\text{Kolmogorov with friction}}{\Delta_v   + <v,\nabla_v > + <v,\nabla_x > - \de_t}$ &   $\omega_{2n}\left(2e^{t} - \frac{t}{2}-1+\frac{t}{2}e^{2t}-e^{2t}\right)^{n}$  & {\color{red}$n$} &  {\color{red}$2n$} &  {\color{red}$4n$}  & {\color{red}$\infty$}  \\ 
\hline \vphantom{$\left(\frac{t^2}{4}+\frac{1}{8}\right)^{\frac{1}{2}}$}$(6^-)$ & $\underset{\text{degenerate Ornstein-Uhlenbeck}}{\Delta_v  - <v,\nabla_v > + <v,\nabla_x > - \de_t}$  &   $\omega_{2n}\left(2e^{-t} + \frac{t}{2}-1-\frac{t}{2}e^{-2t}-e^{-2t}\right)^{n}$  & $-n$  &  $2n$ &  $4n$  & $2n$  \\ 
\hline
\end{tabular}
\centering\caption{\,}\label{fig}
\end{figure}


\vskip 0.2in

\subsection{Ultracontractivity}\label{SS:uc}

We next establish a crucial geometric property of the H\"ormander semigroup that  plays a pervasive role in the remainder of our work. The reader should note that we do not assume \eqref{trace} in Proposition \ref{P:Koneinfty}. As a consequence, such result alone does not imply a decay of the semigroup. In this respect, see Corollary \ref{C:ultra}.

\begin{proposition}[$L^p\to L^\infty$ Ultracontractivity]\label{P:Koneinfty}
Let $1\le p<\infty$ and $f\in \Lp$. For every $X\in \RN$ and $t>0$ we have
\[
|P_t f(X)| \le \frac{c_{N,p}}{V(t)^{1/p}} ||f||_{p},
\]
for a certain constant $c_{N,p}>0$.
\end{proposition}

\begin{proof}
Applying H\"older's inequality to $P_t f(X) = \int_{\R^N} p(X,Y,t) f(Y) dY$, we find 
\[
|P_t f(X)| \le ||f||_p \left(\int_{\RN} p(X,Y,t)^{p'} dY\right)^{\frac 1{p'}} ,
\]
with $1/p + 1/{p'} = 1$. Using \eqref{g} it is now easy to recognise that for any $1\le r < \infty$, there exists a universal constant $c_{N,r}>0$ such that 
\begin{equation}\label{pY}
\left(\int_{\RN} p(X,Y,t)^r dY\right)^{\frac 1r} = \frac{c_{N,r}}{V(t)^{1-\frac 1r}}.
\end{equation}
The desired conclusion now follows taking $r = p'$ in \eqref{pY}. 

\end{proof}

For later use, we also record the following formula, dual to \eqref{pY}, which easily follows by  \eqref{gstar}
$$\left(\int_{\RN} p(X,Y,t)^r dX\right)^{\frac 1r} = \frac{c_{N,r} e^{-t \frac{\operatorname{tr} B}r}}{V(t)^{1-\frac 1r}}.$$

\begin{corollary}\label{C:ultra}
Assume \eqref{trace} and let $1\le p <\infty$. For every $f\in \Lp$ and $X\in \RN$, we have
\[
\underset{t\to \infty}{\lim}\ |P_t f(X)| = 0.
\]
\end{corollary}

\begin{proof}
By Proposition \ref{P:Koneinfty} we  have for every $X\in \RN$ and $t>0$
\[
|P_t f(X)| \le \frac{c_{N}}{V(t)^{1/p}}\ ||f||_p.
\]
Combining this estimate with Corollary \ref{C:infty} we find
\[
\underset{t\to \infty}{\lim}\ |P_t f(X)| \le c_N ||f||_p\  \underset{t\to \infty}{\lim}\ \frac{1}{V(t)^{1/p}} = 0.
\] 

\end{proof}


\section{Sobolev spaces}\label{S:fpA}

In the recent work \cite{GT} we developed a fractional calculus for the operators $\K$ in \eqref{K0} and solved the so-called \emph{extension problem}. This is a generalisation of the famous work by Caffarelli and Silvestre for the fractional Laplacian $(-\Delta)^s$, see \cite{CS}.  As a by-product of our work, we obtained a nonlocal calculus for the ``time-independent" part of the operators $\K$, namely the second order partial differential operator
$$\sA u = \operatorname{tr}(Q \nabla^2 u) + <BX,\nabla u>.$$
It is worth mentioning here that boundary values for these elliptic-parabolic operators were studied by Fichera in his pioneering works \cite{F1}, \cite{F2}.

Since the nonlocal operators $\As$ play a central role in the present work we now recall their definition from \cite[Definition 3.1]{GT}. Hereafter, when considering the action of the operators $\sA$ or $\As$ on a given $\Lp$, the reader should keep in mind  Remark \ref{R:id}.

\begin{definition}\label{D:flheat}
Let $0<s<1$. For any $f\in \So$ we define the nonlocal operator $\As$ by the following pointwise formula
\begin{align}\label{As}
(-\mathscr A)^s f(X) & =  - \frac{s}{\G(1-s)} \int_0^\infty t^{-(1+s)} \left[P_t f(X) - f(X)\right] dt,\qquad X\in\RN.
\end{align}
\end{definition}

We mention that it was shown in \cite{GT} that the right-hand side of \eqref{As} is a convergent integral (in the sense of Bochner) in $L^\infty$, and also in $L^p$ for any $p\in [1,\infty]$ when \eqref{trace} holds. We note that, when $\sA = \Delta$, it is easy to see that formula \eqref{As} allows to recover M. Riesz' definition in \cite{R} of the fractional powers of the Laplacian
\[
(-\Delta)^s f(X) = \frac{s 2^{2s-1} \G\left(\frac{N+ 2s}{2}\right)}{\pi^{\frac N2} \G(1-s)}\int_{\RN} \frac{2 u(X) - u(X+Y) - u(X-Y)}{|Y|^{N+2s}} dY.
\]
Definition \eqref{As} comes from Balakrishnan's seminal work \cite{B}. 
The nonlocal operators \eqref{As} enjoy the following semigroup property (see \cite{B} for the case $s+s'<1$ and \cite{GT2} for $s+s'=1$).

\begin{proposition}\label{P:bala2}
Let $s, s'\in (0,1)$ and suppose that $s+s'\in (0,1]$. Then, for every $f\in \So$ we have
\[
(-\mathscr A)^{s+s'} f = (-\mathscr A)^s \circ (\mathscr A)^{s'} f.
\]
\end{proposition}
For any given $1\le p<\infty$, and any $0<s<1$, we denote by 
$$D_{p,s} = \{f\in L^p\mid \As f \in L^p\},$$
the domain of $\As$ in $L^p$. The operator $\As$ can be extended to a closed operator on its domain, see \cite[Lemma 2.1]{B}. Therefore, endowed with the graph norm
$$||f||_{D_{p,s}} \overset{def}{=} ||f||_{p} + ||(-\sA)^s f||_{p},$$
$D_{p,s}$ becomes a Banach space. 
The next lemma shows that, when \eqref{trace} holds, then
$\So\ \subset \ D_{p,s}$.

\begin{lemma}\label{L:inclusion}
Assume \eqref{trace}, and let $0<s<1$. Given $1\le p \le \infty$, one has 
\[
(-\sA)^s(\So) \subset L^p.
\]
\end{lemma}

\begin{proof}
In view of \eqref{As} we have
\begin{align*}
||(-\mathscr A)^s f||_{p} & \le   \frac{s}{\G(1-s)} \int_0^\infty t^{-{(1+s)}} ||P_t f - f||_{p} dt
\\
& = \frac{s}{\G(1-s)} \left\{\int_0^1 t^{-{(1+s)}} ||P_t f - f||_{p} dt + \int_1^\infty t^{-{(1+s)}} ||P_t f - f||_{p} dt\right\}.
\end{align*}
Thanks to Lemma \ref{L:Lprate} we now have for some universal constant $C>0$,
\[
\int_0^1 t^{-{(1+s)}} ||P_t f - f||_{p} dt \le C ||\mathscr A f||_{p} \int_0^1 \frac{dt}{t^s} < \infty.
\]
On the other hand, by (iv) in Lemma \ref{L:Pt} we know that, under the hypothesis \eqref{trace}, $P_t$ is a contraction in $L^p$. We thus obtain
\[
\int_1^\infty t^{-(1+s)} ||P_t f - f||_{p} dt \le 2 ||f||_{p} \int_1^\infty \frac{dt}{t^{1+s}} < \infty.
\]
This proves the desired conclusion.

\end{proof}

We now use the nonlocal operators $\As$ to introduce the functional spaces naturally attached to the operator $\sA$. These spaces involve a fractional order of differentiation that is intrinsically calibrated both on the directions of ellipticity of the second order part of \eqref{K0}, as well as on the drift.

\begin{definition}[Sobolev spaces]\label{D:sobolev}
Assume \eqref{trace}, and let $1\le p < \infty$ and $0<s<1$. We define the Sobolev space as $\Lo = \overline{\So}^{|| \  ||_{D_{p,s}}}$. 
\end{definition}

\begin{remark}\label{R:sob}
Some comments are in order:
\begin{itemize}
\item[(i)] the space $\Lo$ is a Banach subspace of $D_{p,s}$. It is non-trivial since in view of Lemma \ref{L:inclusion} we have $\So \subset \Lo$.
\item[(ii)] The second (important) remark is that when $Q = I_N$ and $B = O_N$, and thus $\sA = \Delta$, then for $1<p<\infty$ and $s = 1/2$ the space $\Lo$ coincides with the classical Sobolev space 
$
W^{1,p} = \{f\in L^p\mid \nabla f\in L^p\},
$
endowed with the usual norm
$||f||_{W^{1,p}} =  ||f||_{L^p} + ||\nabla f||_{L^p}$.
In other words, one has $\mathscr L^{1,p} = W^{1,p}$, for $1<p<\infty$.
This follows from the well-known fact that 
$W^{1,p} = \overline{\So}^{|| \  ||_{W^{1,p}}}$ (Friedrich's mollifiers, see \cite{Fri}),
combined with the $L^p$ continuity of the singular integrals (Riesz transforms) in the range $1<p<\infty$, see \cite[Ch. 3]{Stein}. This implies the double inequality
\[
A_p \| \Delta^{1/2} f \|_p \le \|  \nabla f  \|_p \le B_p \| \Delta^{1/2} f \|_p,\ \ \ \ \ f \in \So.
\] 
\item[(iii)] We mention that such inequality, and therefore the identity $\mathscr L^{1,p} = W^{1,p}$, continue to be valid on any complete Riemannian manifold with Ricci lower bound $\operatorname{Ric} \ge - \kappa$, where $\kappa \ge 0$. This was proved by Bakry in \cite{Bakry}. A generalisation to the larger class of sub-Riemannian manifolds with transverse symmetries was subsequently obtained in \cite{BG}. \item[(iv)] As a final comment we note that, when $p=2$, and again $\sA = \Delta$, then the space $\Lo$ coincides with the classical Sobolev space of fractional order $H^{2s}$, see e.g. \cite{LM} or \cite{Ad}.
\end{itemize}
\end{remark}

We close this section by recalling the result from \cite{GT} that will be needed in the next one. Given $0<s<1$, let $a=1-2s$. The extension problem for $\As$ consists in the following degenerate Dirichlet problem in the variables $(X,z)\in \Rnp$, where $X\in \RN$ and $z>0$:
\begin{equation}\label{ep}
\begin{cases}
\sA_a U \overset{def}{=} \sA U + \p_{zz} U + \frac az \p_z U = 0,\ \ \ \ \ \ \ \ \ \ \ \ \ \ \ \text{in}\ \Rnp,
\\
U(X,0) = f(X)\ \ \ \ \ \ \ X\in \R^N,
\end{cases}
\end{equation}
where $f\in\So$. We note that, since $s\in (0,1)$, the relation $a = 1-2s$ gives $a \in (-1,1)$, and that, in particular, $a = 0$ when $s = 1/2$. For the following Poisson kernel for the problem \eqref{ep}, and for the subsequent Theorem \ref{T:extLinfty}, one should see \cite[Def. 5.1 and Theor. 5.5]{GT}, 
\begin{equation}\label{pk}
\Po^{(a)}(X,Y,z) =   \frac{1}{2^{1-a} \G(\frac{1-a}{2})} \int_0^\infty \frac{z^{1-a}}{t^{\frac{3-a}{2}}} e^{-\frac{z^2}{4t}}  p(X,Y,t)  dt,\qquad X,Y\in\RN,\,\, z>0.
\end{equation} 

The next result generalises the famous one by Caffarelli and Silvestre in \cite{CS} for the nonlocal operator $(-\Delta)^s$. 

\begin{theorem}\label{T:extLinfty}
The function 
$U(X,z) = \int_{\R^{N}} \Po^{(a)}(X,Y,z) f(Y) dY,$
belongs to $C^\infty(\RN\times (0,\infty))$ and solves the extension problem \eqref{ep}. By this we mean that $\sA_a U = 0$ in $\Rnp$, and we have in $L^\infty$
\begin{equation}\label{ULinfty}
\underset{z\to 0^+}{\lim} U(\cdot;z) = f.
\end{equation}
Moreover, we also have in $L^\infty$
\begin{equation}\label{nconvAinfty}
- \frac{2^{-a} \Gamma\left(\frac{1-a}2\right)}{\Gamma\left(\frac{1+a}2\right)}  \underset{z\to 0^+}{\lim} \paa U(\cdot,z) = (-\mathscr A)^s f.
\end{equation}
If furthermore one has $\operatorname{tr} B\ge 0$, then the convergence in \eqref{ULinfty}, \eqref{nconvAinfty} is also in $L^p$ for any $1\le p\le \infty$.
\end{theorem}



\section{The key Littlewood-Paley estimate}\label{S:poisson}

In the Hardy-Littlewood theory the weak $L^1$ continuity of the maximal function occupies a central position. It is natural to expect that such result play a similar role for the operators in the general class \eqref{K0}, but because of the intertwining of the $X$ and $t$ variables it is not obvious how to select a ``good" maximal function. At first it seems natural to consider $\Ma f(X) = \underset{t>0}{\sup}\ |P_t f(X)|,$
but such object presents an obstruction connected with the mapping properties of the Littlewood-Paley function that controls it.
We have been able to circumvent this difficulty by combining a far-reaching idea of E. Stein in \cite{Steinlp} with our work in \cite{GT}. In this respect, the case $s = 1/2$ of Theorem \ref{T:extLinfty} provides the main technical tool to bypass the above mentioned difficulties connected with $P_t$. It will lead us to Theorem \ref{T:maximal}, which is the main result of this section. 

Since in what follows we are primarily interested in the nonlocal operator $(-\sA)^{1/2}$ (the case $a = 0$ in Theorem \ref{T:extLinfty}), we will focus our attention on the corresponding Poisson kernel, which for ease of notation we henceforth denote by 
$\mathscr P(X,Y,z) \overset{def}{=} \Po^{(0)}(X,Y,z)$.
In such case, formula \eqref{pk} reads
\begin{equation}\label{pk0}
\Po(X,Y,z) =   \frac{1}{\sqrt{4\pi}} \int_0^\infty \frac{z}{t^{3/2}} e^{-\frac{z^2}{4t}}  p(X,Y,t)  dt,\qquad X,Y\in\RN,\,\, z>0.
\end{equation} 

\begin{definition}\label{D:ps}
We define the \emph{Poisson semigroup} as follows
\[
\Po_z f(X) = \int_{\RN} \Po(X,Y,z) f(Y) dY,\ \ \ \ \ \ \ \ \ \ f\in \So.
\]
\end{definition}
Using \eqref{pk0} and exchanging the order of integration in the above definition, we obtain the following useful representation of the semigroup $\Po_z$ in terms of the H\"ormander semigroup $P_t$
\begin{equation}\label{altrep1}
\Po_z f(X) = \frac{1}{\sqrt{4\pi}} \int_0^\infty \frac{z}{t^{3/2}} e^{-\frac{z^2}{4t}} P_t f(X) dt.
\end{equation}
This is of course an instance of Bochner's subordination, see \cite{Bo}.
We note in passing that, when the operator $\sA = \Delta$, from \eqref{altrep1} we recover the classical Poisson kernel for the half-space $\Rnp$, see \cite[(15), p.61]{Stein},
\[
\Po(X,Y,z) = \frac{\G(\frac{N+1}{2})}{\pi^{\frac{N+1}{2}}} \frac{z}{(z^2+|X-Y|^2)^{\frac{N+1}{2}}}.
\]
Some basic facts that we need about $\{\Po_z\}_{z>0}$ are contained in the next result.

\begin{lemma}\label{L:if}
The following properties hold:
\begin{itemize}
\item[(i)] For every $X\in \RN$ and $z>0$ we have $\Po_z 1(X) = 1;$
\item[(ii)] $\Po_z:L^\infty \to L^\infty$ with $||\Po_z||_{\infty\to \infty} \le 1$;
\item[(iii)] let $1\le p<\infty$. If \eqref{trace} holds, then $\Po_z:L^p \to L^p$ with $||\Po_z||_{p\to p} \le 1$;
\item[(iv)] let $f\in \So$. Then, 
$\underset{z\to 0^+}{\lim} \frac{\Po_z f(X) - f(X)}{z} = (-\sA)^{1/2} f(X);$
\item[(v)] The function $U(X,z) = \Po_z f(X)$ belongs to $C^\infty(\Rnp)$ and it satisfies the partial differential equation $\p_{zz} U + \sA U = 0$. Moreover, $\underset{z\to 0^+}{\lim}\ U(\cdot,z) = f$ in $L^\infty$ and in $L^p$ for every $1\le p < \infty$, when \eqref{trace} holds. 
\end{itemize}
\end{lemma}

\begin{proof}
The proof of (i) follows by taking $a=0$ in \cite[Proposition 5.2]{GT}. (ii) is a direct consequence of (i). To establish (iii) we use \eqref{altrep1}, that gives
\[
||\Po_z f||_p \le \frac{1}{\sqrt{4\pi}} \int_0^\infty \frac{z}{t^{3/2}} e^{-\frac{z^2}{4t}} ||P_t f||_p  dt \le \frac{||f||_p}{\sqrt{4\pi}} \int_0^\infty \frac{z}{t^{3/2}} e^{-\frac{z^2}{4t}}  dt = ||f||_p, 
\]
where in second inequality we have used (iv) in Lemma \ref{L:Pt}, and in the last equality the fact that
\[
\frac{1}{\sqrt{4\pi}} \int_0^\infty \frac{z}{t^{3/2}} e^{-\frac{z^2}{4t}}  dt = 1.
\]
The properties (iv) and (v) follow from the case $a = 0$ of Theorem \ref{T:extLinfty}. 

\end{proof}

\begin{remark}\label{R:poisson}
We note explicitly that (iv) in Lemma \ref{L:if} says, in particular, that the infinitesimal generator of $\Po_z$ is the nonlocal operator $(-\sA)^{1/2}$, i.e., $\Po_z = e^{z\sqrt{-\sA}}$. In the case when $\sA = \Delta$ one should see the seminal work \cite{Tai}, where an extensive use of the Poisson semigroup was made in connection with smoothness properties of functions. 
\end{remark}

Given a reasonable function $f$ (for instance, $f\in \So$) we now introduce its \emph{Poisson radial maximal function} as follows 
\begin{equation}\label{pm}
\Ma^\star f(X) \overset{def}{=} \underset{z>0}{\sup}\ |\Po_z f(X)|,\ \ \ \ X\in \RN.
\end{equation}

\begin{lemma}\label{L:improvement}
There exists a universal constant $A>0$ such that
\begin{equation}\label{maxpoisson}
\Ma^\star f(X) \le A\ \underset{t>0}{\sup}\ \left|\frac 1t \int_0^t P_s f(X) ds\right|.
\end{equation}
\end{lemma}

\begin{proof}
Adapting an idea idea in  \cite[p. 49]{Steinlp}, we can write \eqref{altrep1} as
 \begin{equation}\label{altrep2}
\Po_z f(X) = \int_0^\infty g(z,t) \frac{d}{dt} (t F(t)) dt,
\end{equation}
where $g(z,t) =  \frac{z t^{-3/2}}{\sqrt{4\pi}} e^{-\frac{z^2}{4t}}$, and we have let $F(t) = \frac 1t \int_0^t P_s f(X) ds$. Notice that by (ii) in Lemma \ref{L:Pt}, we can bound $|F(t)| \le ||f||_\infty$. Also observe that $t g(z,t)\to 0$ as $t\to \infty$, and  that $t\to t \big|\frac{\p g}{\p t}(z,t)\big|\in L^1(0,\infty)$. We can thus integrate by parts in \eqref{altrep2}, obtaining
\[
|\Po_z f(X)| = \left|\int_0^\infty t \frac{\p g}{\p t}(z,t) F(t) dt\right| \le A(z)\ \underset{t>0}{\sup}\ \left|\frac 1t \int_0^t P_s f(X) ds\right|,
\]
 with
\[
A(z) = \int_0^\infty t \big|\frac{\p g}{\p t}(z,t)\big| dt < \infty.
\]
To complete the proof it suffices to observe that $A(z) \le A = 7/2$ for every $z>0$. This follows from the fact that $t \frac{\p g}{\p t}(z,t) = \big(\frac{z^2}{t} - \frac 32\big) g(z,t)$, and that $\int_0^\infty g(z,t) dt = 1$, and $\int_0^\infty \frac{z^2}{t} g(z,t) dt = 2$.

\end{proof}

The next is the main result in this section. It provides the key maximal theorem for the class \eqref{K0}. As far as we know, such tool has so far been missing in the existing literature.

\begin{theorem}\label{T:maximal}
Assume \eqref{trace}. Then, the Poisson maximal function \eqref{pm} enjoys the following properties: (a) there exists a universal constant $A>0$ such that, given $f\in L^1$, for every $\la>0$ one has
\[
|\{X\in \RN \mid \Ma^\star f(X) > \la\}| \le \frac{2A}\la ||f||_{L^1};
\]
(b) let $1<p\le \infty$, then there exists a universal constant $A_p>0$ such that for any $f\in L^p$ one has
\[
||\Ma^\star f||_{L^p} \le A_p ||f||_{L^p}.
\]
\end{theorem}

\begin{proof}
(a) In view of (iv) in Lemma \ref{L:Pt}, we know that $\{P_t\}_{t>0}$ is contractive in $L^1$ and in $L^\infty$. Furthermore, by Corollary \ref{C:Ptpzero} it is a strongly continuous semigroup in $L^p$, for every $1\le p<\infty$. We can thus apply the powerful Hopf-Dunford-Schwartz ergodic theorem, see \cite[Lemma 6, p. 153]{DS}, and infer that, if $f\in L^1$, then for every $\la>0$ one has
\begin{equation}\label{ergodic}
|\{X\in \RN\mid f^\star(X)> \la\}| \le \frac 2\lambda \int_{\{X\in \RN\mid |f(X)|> \la/2\}} |f(X)| dX \le \frac 2\lambda ||f||_1,
\end{equation}
where we have let 
\[
f^\star(X) \overset{def}{=} \underset{t>0}{\sup}\ \left|\frac 1t \int_0^t P_s f(X) ds\right|.
\]
On the other hand, \eqref{maxpoisson} in Lemma \ref{L:improvement} gives
\begin{align*}
|\{X\in \RN\mid \Ma^\star f(X)> \la\}| \le |\{X\in \RN\mid f^\star(X)> \la/A\}| \le \frac{2A}\la ||f||_1,
\end{align*}
where in the second inequality we have used \eqref{ergodic}.

(b) We observe that from (ii) in Lemma \ref{L:if} we trivially have 
\[
\Ma^\star : L^\infty\ \longrightarrow\ L^\infty,\ \ \ \ \ \text{with}\ \ \ \ \  ||\Ma^\star||_{L^\infty\to L^\infty} \le 1.
\]
By (a) and the theorem of real interpolation of Marcinckiewicz (see \cite[Chap. 1, Theor. 5]{Stein}), we conclude that (b) is true for some $A_p>0$.

\end{proof}


\section{The fractional integration operator $\Ia$}\label{S:fi}

In the classical theory of Hardy-Littlewood-Sobolev the M. Riesz'  operator of fractional integration plays a pivotal role. We recall, see \cite{R} and also \cite[Chap. 5]{Stein}, that given a number $0<\alpha<N$, the latter is defined by the formula
\begin{equation}\label{pot}
I_\alpha f(X) = \frac{\G(\frac{N-\alpha}2)}{2^\alpha \pi^{\frac N2} \G(\frac{\alpha}2)} \int_{\RN} \frac{f(Y)}{|X-Y|^{N-\alpha}} dy.
\end{equation}
The essential feature of such operator is that it provides the inverse of the fractional powers of the Laplacian, in the sense that for any $f\in \So$ one has
$f = I_\alpha \circ (-\Delta)^{\alpha/2} f$. Its role in the Hardy-Littlewood theory is perhaps best highlighted by the following interpolating inequality which goes back to \cite[Chapter 5]{Stein}, see also \cite{He}. Suppose $1\le p<n/\alpha$ and that $f\in L^p$. Then, one has for any $\ve>0$,
\begin{equation}\label{keyes}
|I_\alpha f(x)| \le C(n,\alpha,p)\left(Mf(x) \ve^{\alpha} + ||f||_p\ \ve^{-(\frac np - \alpha)}\right).
\end{equation}
The usefulness of the inequality \eqref{keyes} is multi-faceted. One the one hand, when $p>1$, combined with the strong $\Lp$ continuity of the maximal operator, it shows that $I_\alpha : \Lp \to L^q$, provided that $1/p - 1/q = \alpha/n$. On the other hand, \eqref{keyes}  allows to immediately establish the geometric weak end-point result $W^{1,1}\  \hookrightarrow\ L^{\frac{n}{n-1},\infty}$. 
This implies, in turn, the isoperimetric inequality $P(E) \ge C_n |E|^{\frac{n}{n-1}}$ and, equivalently, the strong geometric Sobolev embedding, $BV\  \hookrightarrow\ L^{\frac{n}{n-1}}$,
where $P(E)$ denotes De Giorgi's perimeter and $BV$ the subspace of $L^1$ of functions with bounded variation (for these aspects we refer to \cite{CDGcag}, where these ideas were developed in the general framework of Carnot-Carth\'eodory spaces).

In this section, we use the Poisson semigroup $\{\Po_z\}_{z>0}$ in Definition \ref{D:ps} to introduce, in our setting, the counterpart of the potential operators \eqref{pot}, see Lemma \ref{L:altexIa}.  Theorem \ref{T:inverse} is the first main result of the section. It shows that the operator $\mathscr I_{2s}$ inverts the nonlocal operator $\As$. In the next definition the reader needs to keep in mind the number $D_\infty$ in Definition \ref{D:hld}.

\begin{definition}\label{D:fi}
Let $0< \alpha < D_\infty$. Given $f\in \So$, we define the \emph{Riesz potential} of order $\alpha$ as follows
\[
\Ia f(X) = \frac{1}{\G(\alpha/2)} \int_0^\infty t^{\alpha/2 - 1} P_t f(X) dt.
\]
\end{definition}

Let us observe that for every $X\in \RN$ the integral in Definition \ref{D:fi} converges absolutely. To see this we write 
\[
\int_0^\infty t^{\alpha/2 - 1} P_t f(X) dt = \int_0^1 t^{\alpha/2 - 1} P_t f(X) dt + \int_1^\infty t^{\alpha/2 - 1} P_t f(X) dt.
\]
The integral on $[0,1]$ is absolutely convergent for any $\alpha>0$ since, using (ii) in Lemma \ref{L:Pt}, we can bound $|P_t f(X)|\le ||P_t f||_\infty \le ||f||_\infty$. For the integral on $[1,\infty)$ we use the ultracontractivity of $P_t$ in Proposition \ref{P:Koneinfty}, which gives for any $X\in \RN$ and $t>0$,
\[
\int_1^\infty t^{\alpha/2 - 1} |P_t f(X)| dt \le c_{N} ||f||_1 \int_1^\infty  \frac{t^{\alpha/2 - 1}}{V(t)} < \infty,
\]
since $0<\alpha < D_\infty$. In the next lemma, using Bochner's subordination, we recall a useful alternative expression of the potential operators $\Ia$ based on the Poisson semigroup $\{\Po_z\}_{z>0}$. 

\begin{lemma}\label{L:altexIa}
Let $0<\alpha<D_\infty$. For any $f\in \So$ one has
\[
\Ia f(X) = \frac{1}{\G(\alpha)} \int_0^\infty z^{\alpha - 1} \Po_z f(X) dz.
\]
\end{lemma}

\begin{proof}
We have from \eqref{altrep1}
\begin{align*}
& \int_0^\infty z^{\alpha - 1} \Po_z f(X) dz = \frac{1}{\sqrt{4\pi}} \int_0^\infty z^{\alpha - 1} \int_0^\infty \frac{z}{t^{3/2}} e^{-\frac{z^2}{4t}} P_t f(X) dt dz
\\
& = \frac{1}{\sqrt{4\pi}} \int_0^\infty \frac{1}{t^{3/2}}\left(\int_0^\infty z^{\alpha+1} e^{-\frac{z^2}{4t}} \frac{dz}z\right) P_t f(X) dt 
\\
& = \frac{2^{\alpha-1}\G(\frac{\alpha+1}{2})}{\sqrt \pi} \int_0^\infty t^{\alpha/2 -1} P_t f(X) dt = \frac{2^{\alpha-1}\G(\frac{\alpha+1}{2})\G(\alpha/2)}{\sqrt \pi}\ \Ia f(X) 
\\
& = \G(\alpha) \Ia f(X), 
\end{align*}
where in the last equality we have used, with $x = \alpha/2$, the well-known duplication formula for the gamma function $2^{2x-1} \G(x) \G(x+1/2) = \sqrt \pi \G(2x),$
see e.g. \cite[formula (1.2.3)]{Le}.

\end{proof}

The next basic result plays a central role for the remainder of this paper. It shows that the integral operator $\Ia$ is the inverse of the nonlocal operator $(-\sA)^{\alpha/2}$.

\begin{theorem}\label{T:inverse}
Suppose that \eqref{trace} hold, and let $0<s<1$. Then, for any $f\in \So$ we have
\[
f = \mathscr I_{2s} \circ \As f = \As \circ \mathscr I_{2s} f.
\]
\end{theorem}

\begin{proof}
We only prove the first equality, the second is established similarly. It will be useful in what follows to adopt the following alternative expression, see \cite{B}, of the nonlocal operator \eqref{As}
\begin{align}\label{bala2}
(-\mathscr A)^s f(X) & = \frac{\sin(\pi s)}{\pi} \int_{0}^{\infty} \lambda^{s-1} R(\lambda, \mathscr A) (-\mathscr A)f(X) d\la 
\\
& = \frac{\sin(\pi s)}{\pi} \int_{0}^{\infty} \la^{s-1} (I - \la R(\la,\sA)) f(X) d\la,
\notag
\end{align}
where we have denoted by $R(\lambda, \mathscr A)= (\la I - \mathscr A)^{-1}$ the resolvent of $\sA$ in $L^\infty_0$ (we are now identifying $\sA$ with $\sA_\infty$, the infinitesimal generator of $\{P_t\}_{t>0}$ in $L^\infty_0$, see Remarks \ref{R:infty}, \ref{R:id} and Lemma \ref{L:specter}). We remark that either one of the integrals in the right-hand side of \eqref{bala2} converge in $\Li$. For instance, in the first integral there is no issue near $\la = 0$ since $s>0$, whereas (3) in Lemma \ref{L:specter} gives $\lambda^{s-1} ||R(\lambda, \mathscr A) (-\mathscr A)f||_\infty \le \lambda^{s-2} ||\mathscr A f||_\infty$, which is convergent near $\infty$. Keeping in mind that by (2) in Lemma \ref{L:specter} we have $R(\la,\sA) f = \int_0^\infty e^{-\la t} P_t f dt$,
we can alternatively express \eqref{bala2} as follows
\begin{equation}\label{bala3}
(-\mathscr A)^s f(X)  = \frac{\sin(\pi s)}{\pi} \int_0^\infty \int_{0}^{\infty} \lambda^{s-1} e^{-\la \tau} P_\tau (-\mathscr A)f(X) d\la d\tau.
\end{equation}
If we now combine Definition \ref{D:fi} with \eqref{bala3}, we find
\begin{align*}
 \mathscr I_{2s}\left((-\mathscr A)^s f\right)(X) & =\frac{\sin(\pi s)}{\pi\G(s)} \int_0^\infty\left(\int_{0}^{\infty}{\lambda^s e^{-\lambda\tau} \left( \int_{0}^{\infty}t^sP_{t+\tau}(-\mathscr A)f(X)\frac{dt}{t}\right)}\frac{d\lambda}{\lambda}\right)d\tau \\
&=\frac{\sin(\pi s)}{\pi\G(s)} \int_0^\infty\int_{0}^{\infty} u^s P_{\tau(1+u)}(-\mathscr A) f(X)\left(\int_{0}^{\infty}{\lambda^s e^{-\lambda\tau}\tau^s \frac{d\lambda}{\lambda}}\right)\frac{du}{u}d\tau\\
& =\frac{\sin(\pi s)}{\pi} \int_0^\infty u^s \int_{0}^{\infty}  P_{\tau(1+u)}(-\mathscr A) f(X)\,d\tau\frac{du}{u}
\\
& = - \frac{\sin(\pi s)}{\pi} \int_0^\infty \frac{u^{s-1}}{1+u} du \int_0^\infty P_\rho  \sA f(X) d\rho,
\end{align*}
where in the innermost integral we have made the change of variables $\rho= \tau (1+ u)$. We notice that one can justify the above relations by a standard application of Fubini and Tonelli theorems once we recognize that, for large $t$, the ultracontractivity and the fact that $D_\infty\geq 2>2s$ ensure the right summability properties. We now make the key observation that (b) and (c) in Lemma \ref{L:invS} and the assumption \eqref{trace} imply, in view of Corollary \ref{C:ultra},
\[
\int_0^\infty \sA P_\rho f(X) d\rho = \int_0^\infty \frac{d}{d\rho} P_\rho f(X) d\rho = - f(X).
\]
In order to reach the desired conclusion we are only left with observing, see e.g. \cite[3.123 on p.105]{T}, that $\int_0^\infty \frac{u^{s-1}}{1+u} du = \Gamma(s)\Gamma(1-s)=\frac{\pi}{\sin(\pi s)}$. 
 
\end{proof}

 
 \section{An intrinsic embedding theorem of Sobolev type}\label{S:sobolev}
 
In this section we prove our main embedding of Sobolev type, Theorem \ref{T:sob}. Our strategy follows the classical approach to the subject. We first establish the key Hardy-Littlewood-Sobolev type result, Theorem \ref{T:main}. With such tool in hands, we are easily able to obtain the Sobolev embedding, Theorem \ref{T:sob}. We note that these results do not tell the whole story since, as noted in Remark \ref{R:twods}, their main assumption \eqref{vol} implies necessarily that $D_0\le D_\infty$. But we have seen in Ex.4 in fig.\ref{fig} that there exist operators of interest in physics for which we have instead $D_0>D_\infty$. These cases are handled by Theorems \ref{T:mainsum} and \ref{T:sobsum}. Since we will need to have in place all the results from the previous sections, hereafter we assume without further mention  that the assumption \eqref{trace}
be in force. Our first result shows a basic property of the Poisson semigroup. 

\begin{lemma}[Ultracontractivity of $\Po_z = e^{z \sqrt{-\sA}}$]\label{L:Pzuc}
Suppose that there exist numbers $D, \gamma_D>0$ such that for every $t>0$ one has
\begin{equation}\label{vol}
V(t) \ge \gamma_D\ t^{D/2}.
\end{equation}
If $1\le p<\infty$ one has for $f\in \Lp$, $X\in \RN$ and any $z>0$,
\[
|\Po_z f(X)|\le \frac{C_1}{z^{D/p}}\ ||f||_p,
\]
where $C_1 = C_1(N,D,p)>0$. 
\end{lemma}

\begin{proof}
From \eqref{altrep1}, Proposition \ref{P:Koneinfty} and \eqref{vol} we find 
 \begin{align*}
|\Po_z f(X)| & \le \frac{c_{N,p}}{\sqrt{4\pi}} ||f||_p \int_0^\infty \frac{z}{t^{1/2}V(t)^{1/p}} e^{-\frac{z^2}{4t}} \frac{dt}t 
 \le C ||f||_p \int_0^\infty \frac{z}{t^{\frac{1}{2}(\frac Dp + 1)}} e^{-\frac{z^2}{4t}} \frac{dt}t  =  C_1\ ||f||_p\ z^{-D/p},
\end{align*}
where $C_1 = C_1(N,D,\gamma_D,p)>0$.

\end{proof}

\begin{remark}\label{R:twods}
Keeping Definitions \ref{D:D0} and \ref{D:hld} in mind, the reader should note that the assumption \eqref{vol} implies necessarily that $D_0 \le D\le D_\infty$. Thus, the case $D_0>D_\infty$ is left out, but it will be addressed in Theorems \ref{T:mainsum} and \ref{T:sobsum}.
\end{remark}

The next proposition contains an essential interpolation estimate which generalises to the degenerate non-symmetric setting of \eqref{K0} the one in \cite{V85}, see also \cite{VSC}. Such tool represents the semigroup replacement of the Stein-Hedberg estimate \eqref{keyes}.

\begin{proposition}\label{P:keyest}
Assume \eqref{vol}, and let $0<\alpha<D$. Given $1\le p<D/\alpha$ there exist a constant $C_2  = C_2(N,D,\alpha,\gamma_D,p)>0$, such that for every $f\in \So$ and $\ve>0$
\begin{equation}\label{wow}
|\Ia f(X)| \le \frac{1}{\G(\alpha+1)} \Ma^\star f(X)\ \ve^\alpha + C_2\ ||f||_p\ \ve^{\alpha - \frac Dp},
\end{equation}
where $\Ma^\star$ is as in \eqref{pm}.
\end{proposition}

\begin{proof}
We begin by noting that we know from (3) in Remark \ref{R:volumes} that $D_\infty \ge D$, and thus $\Ia$ is well defined for all $0<\alpha < D$. Now, for a given $f\in \So$ using Lemma \ref{L:altexIa} we write  for every $\ve>0$ 
\begin{equation}\label{split}
|\Ia f(X)| \le  \frac{1}{\G(\alpha)} \int_0^\ve z^{\alpha - 1} |\Po_z f(X)| dz + \frac{1}{\G(\alpha)} \int_\ve^\infty z^{\alpha - 1} |\Po_z f(X)| dz.
\end{equation}
The first term is easily controlled by the estimate
\begin{equation}\label{split1}
\frac{1}{\G(\alpha)} \int_0^\ve z^{\alpha - 1} |\Po_z f(X)| dz \le \frac{1}{\G(\alpha+1)} \Ma^\star f(X) \ve^\alpha.
\end{equation} 
Let now $1\le p<D/\alpha$. To control the second term we use Lemma \ref{L:Pzuc} to  obtain
$$\frac{1}{\G(\alpha)} \int_\ve^\infty z^{\alpha - 1} |\Po_z f(X)| dz \le \frac{C_1}{\G(\alpha)} ||f||_p \int_\ve^\infty z^{\alpha - D/p - 1} dz = C_2\ ||f||_p\ \ve^{\alpha - \frac Dp},$$
where $C_2 = C_2(N,D,\alpha,\gamma_D,p)>0$. Combining this estimate with \eqref{split1} and \eqref{split}, we conclude that \eqref{wow} holds. 

\end{proof}

With Proposition \ref{P:keyest} in hands, we can now establish the first main result of this section.

\begin{theorem}[of Hardy-Littlewood-Sobolev type]\label{T:main}
Assume that \eqref{vol} be valid. Then, we have:
(i) for every $0<\alpha<D$ the operator $\Ia$ maps $L^1$ into $L^{\frac{D}{D-\alpha},\infty}$. Furthermore, there exists $S_1 = S_1(N,D,\alpha,\gamma_D)>0$ such that for any $f\in L^1$ one has
\begin{equation}\label{weakL1}
\underset{\la>0}{\sup}\ \la\ |\{X\in \RN\mid |\Ia f(X)|>\la\}|^{\frac{D-\alpha}{D}} \le S_1 ||f||_{1};
\end{equation}
(ii) for every $1<p<D/\alpha$ the operator $\Ia$ maps $L^p$ to $L^q$, with $\frac 1p - \frac 1q = \frac{\alpha}D$. Moreover, there exists $S_p = S_p(N,D,\alpha,\gamma_D,p)>0$ such that one has for any $f\in L^p$ 
\begin{equation}\label{strong}
||\Ia f||_{q} \le S_p ||f||_{p}.
\end{equation}
\end{theorem}

\begin{proof}
(i) Suppose $f\in L^1$, with $||f||_1\not= 0$ (otherwise, there is nothing to prove). The estimate \eqref{wow} reads in this case
\begin{equation}\label{wow1}
|\Ia f(X)| \le \frac{1}{\G(\alpha+1)} \Ma^\star f(X)\ \ve^\alpha + C_2\ ||f||_1\ \ve^{\alpha - D}.
\end{equation}
Given $\la>0$ we choose $\ve>0$ such that $C_2 ||f||_1 \ve^{\alpha - D} = \la$.
With such choice we see from \eqref{wow1}
that
\begin{align*}
& |\{X\in \RN\mid |\Ia f(X)|> 2 \la\}| \le |\{X\in \RN\mid \frac{1}{\G(\alpha+1)} \Ma^\star f(X)\ \ve^\alpha > \la\}|
 \le \frac{2A \ve^\alpha}{\la \G(\alpha+1)}  ||f||_1,
\end{align*}
where in the last inequality we have used (a) in Theorem \ref{T:maximal}. Keeping in mind that from our choice of $\ve$ we have 
$\ve^\alpha = \frac{(C_2 ||f||_1)^{\frac{\alpha}{D-\alpha}}}{\la^{\frac{\alpha}{D-\alpha}}}$,
we conclude that \eqref{weakL1} holds. 

To prove (ii), we suppose now that $1<p<D/\alpha$. Minimising with respect to $\ve$ in \eqref{wow} we easily find for some constant $C_3 = C_3(N,D,\alpha,\gamma_D,p)>0$
\begin{equation}\label{wow3}
|\Ia f(X)| \le C_3 \Ma^\star f(X)^{1-\frac{\alpha p}{D}} ||f||_p^{\frac{\alpha p}D}.
\end{equation}
The desired conclusion \eqref{strong} now follows from \eqref{wow3} and from (b) in Theorem \ref{T:maximal}.

\end{proof}

Theorem \ref{T:main} is the keystone on which the second main result of this section leans. Before stating it, we emphasise that in view of (iii) in Remark \ref{R:volumes} we know that $D_\infty \ge 2$. Therefore, if $0<s<1$ then $2s < 2 <D_\infty$.

\begin{theorem}[of Sobolev type]\label{T:sob}
Suppose that \eqref{vol} hold. Let $0<s< 1$. Given $1\leq p<D/2s$ let $q>p$ be such that
$\frac 1p - \frac 1q = \frac{2s}D$.
\begin{itemize}
\item[(a)] If $p>1$ we have $\Lo\ \hookrightarrow\ L^{\frac{pD}{D-2sp}}.$ More precisely, there exists a constant $S_{p,s} >0$, depending on $N,D,s,\gamma_D, p$, such that for any $f\in \So$ one has 
\[
||f||_{q} \le S_{p,s} ||\As f||_p.
\] 
\item[(b)] When $p=1$ we have $\mathscr L^{2s,1}\ \hookrightarrow\ L^{\frac{D}{D-2s},\infty}.$ More precisely, there exists a constant $S_{1,s} >0$, depending on $N,D,s,\gamma_D$, such that for any $f\in \So$ one has
\[
\underset{\la>0}{\sup}\ \la |\{X\in \RN\mid |f(X)| > \la\}|^{1/q} \le S_{1,s} ||\As f||_{1}.
\]
\end{itemize}
\end{theorem}

\begin{proof}
We observe that (3) in Remark \ref{R:volumes} guarantees that $D\le D_\infty$, and therefore $\mathscr I_{2s}$ is well-defined.  At this point, 
the proof is easily obtained by combining Theorem \ref{T:inverse}, which allows to write for every $X\in \RN$
\[
|f(X)| = |\mathscr I_{2s} \As f(X)|,
\]
with Theorem \ref{T:main}. We leave the routine details to the interested reader. 

\end{proof}

From Remark \ref{R:twods} we know that Theorem \ref{T:main} does not cover situations, such as the Kramers' operator in Ex.4 in fig.\ref{fig}, in which $D_0 > D_\infty$. When this happens we have the following substitute result.
In the sequel, when we write $L^{q_1}+L^{q_2}$ we mean the Banach space of functions $f$ which can be written as $f=f_1+f_2$ with $f_1\in L^{q_1}$ and $f_2\in L^{q_2}$, endowed with the norm
$$||f||_{L^{q_1}+L^{q_2}} = \inf_{f=f_1+f_2,\\ f_1\in L^{q_1},\,f_2\in L^{q_2}}{||f_1||_{L^{q_1}}+||f_2||_{L^{q_2}}}.$$

\begin{theorem}\label{T:mainsum}
Suppose there exist $\gamma>0$ such that for every $t>0$ one has
\begin{equation}\label{vol2}
V(t) \ge \gamma \min\{ t^{D_0/2},t^{D_\infty/2} \}.
\end{equation} 
Then, we have: (i) for every $0<\alpha<D_\infty<D_0$ the operator $\Ia$ maps $L^1$ into $L^{\frac{D_0}{D_0-\alpha},\infty} + L^{\frac{D_\infty}{D_\infty-\alpha},\infty}$. Furthermore, there exists $S_1 = S_1(N,D_0,D_\infty,\alpha,\gamma)>0$ such that for any $f\in L^1$ one has
\begin{equation}\label{weakL1sum}
\min\{\underset{\la>0}{\sup}\ \la\ |\{X\mid |\Ia f(X)|>\la\}|^{\frac{D_0-\alpha}{D_0}}, \underset{\la>0}{\sup}\ \la\ |\{X\mid |\Ia f(X)|>\la\}|^{\frac{D_\infty-\alpha}{D_\infty}} \}\le S_1 ||f||_{L^1};
\end{equation}
(ii) for every $1<p<D_\infty/\alpha<D_0/\alpha$ the operator $\Ia$ maps $L^p$ to $L^{q_0}+L^{q_\infty}$, with $\frac 1p - \frac 1q_0 = \frac{\alpha}D_0$ and $\frac 1p - \frac 1q_\infty = \frac{\alpha}D_\infty$. Moreover, there exists $S_p = S_p(N,D_0,D_\infty,\alpha,\gamma,p)>0$ such that one has for any $f\in L^p(\RN)$ 
\begin{equation}\label{strongsum}
||\Ia f||_{L^{q_0}+L^{q_\infty}} \le S_p ||f||_{L^p}.
\end{equation}
\end{theorem}

\begin{proof}
It suffices to prove the statements for $f\in\So$. Let $1\leq p<+\infty$ and $X\in\RN$. From the ultracontractive estimate in Proposition \ref{P:Koneinfty} and \eqref{vol2}, we obtain from \eqref{altrep1},
 \begin{align*}
|\Po_z f(X)| & \le \frac{c_{N,p}}{\sqrt{4\pi}} ||f||_p \int_0^\infty \frac{z}{t^{1/2}V(t)^{1/p}} e^{-\frac{z^2}{4t}} \frac{dt}t 
\\
& \le \frac{c_{N,p}}{\gamma^{\frac{1}{p}} \sqrt{4\pi}} ||f||_p \int_0^\infty \frac{z}{t^{1/2}\min\left\{ t^{D_0/2p},t^{D_\infty/2p}\right\}} e^{-\frac{z^2}{4t}} \frac{dt}t 
\\
&=C(N,p,\gamma) ||f||_p \int_0^\infty \sqrt{u}e^{-u}\max\big\{ \left(\frac{4u}{z^2}\right)^{D_0/2p},\left(\frac{4u}{z^2}\right)^{D_\infty/2p} \big\} \frac{du}u\\
& \leq  C(N,p,\gamma) ||f||_p \max\{ z^{-D_0/p},z^{-D_\infty/p}\}\int_0^\infty \sqrt{u}e^{-u}\max\left\{ \left(4u\right)^{D_0/2p},\left(4u\right)^{D_\infty/2p} \right\} \frac{du}u.
\end{align*}
For any $X\in \RN$ and $z>0$ we have thus proved 
\begin{equation}\label{poissonmaxz}
|\Po_z f(X)|  \le \bar{C}||f||_p \max\{ z^{-D_0/p},z^{-D_\infty/p} \},
\end{equation}
where $\bar{C} = C(N,p,\gamma,D_0,D_\infty)>0$. Next, let $0<\alpha<D_\infty<D_0$ and $1\le p<D_\infty/\alpha<D_0/\alpha$. As in \eqref{split} and \eqref{split1}, for any $X\in\RN$ we have
\begin{equation}\label{wowgeps}
|\Ia f(X)| \le\frac{1}{\G(\alpha+1)} \Ma^\star f(X) \ve^\alpha + \frac{1}{\G(\alpha)} \int_\ve^\infty z^{\alpha - 1} |\Po_z f(X)| dz.
\end{equation}
To bound the second integral we use \eqref{poissonmaxz} to find
$$\int_\ve^\infty z^{\alpha - 1} |\Po_z f(X)| dz\leq \bar{C}||f||_p\ g(\ve),$$
where
$$g(\ve)=\int_\ve^\infty h(z)\,dz=\int_\ve^\infty z^{\alpha - 1} \max\{ z^{-D_0/p},z^{-D_\infty/p} \}dz.$$
To see that $g(\ve)<\infty$ for all $\ve>0$, it suffices to look at $g(1)$:
$$g(1)=\int_1^\infty z^{\alpha - 1} \max\{ z^{-D_0/p},z^{-D_\infty/p} \}dz=\int_1^\infty z^{\alpha - 1 -D_\infty/p}\,dz<\infty$$
since we have assumed $p<\frac{D_\infty}{\alpha}$. Therefore, $g(\ve)$ is well-defined, $g\in C^1(0,\infty)$, and 
\[
g'(\ve)=-\ve^{\alpha-1}\max\{ \ve^{-D_0/p},\ve^{-D_\infty/p} \}<0,
\]
 which shows that $g$ is strictly decreasing. We also know that $\lim_{\ve\rightarrow+\infty} g(\ve)=0$. Furthermore, if $0<\ve<1$, then
$$g(\ve)=\int_\ve^1 h(z)\,dz + \int_1^{\infty} h(z)\,dz=\int_\ve^1 z^{\alpha-1-\frac{D_0}{p}}\,dz + g(1)=\frac{\ve^{-\left(\frac{D_0}{p}-\alpha\right)}}{\frac{D_0}{p}-\alpha}-\frac{1}{\frac{D_0}{p}-\alpha}+g(1).$$
We infer that $\lim_{\ve\rightarrow 0^+} g(\ve)=+\infty$. Thus $g:(0,\infty)\rightarrow(0,\infty)$ is invertible.\\
Going back to \eqref{wowgeps} we conclude
\begin{equation}\label{Iageps}
|\Ia f(X)| \le\frac{1}{\G(\alpha+1)} \Ma^\star f(X) \ve^\alpha + \frac{\bar{C}}{\G(\alpha)}||f||_p g(\ve)=:G(\ve) \qquad\forall \ve>0.
\end{equation}
To prove (ii) we look for the minimum of $G$ which is attained at some $\ve$ such that
$$\min\{\ve^{D_0/p},\ve^{D_\infty/p}\}=\bar{C}\frac{||f||_p}{\Ma^\star f(X)}=:A_{f}(X).$$
In other words
$$\ve_{min}=\max\{A_{f}(X)^{p/D_0},A_{f}(X)^{p/D_\infty}\}.$$
Going back to \eqref{Iageps} we conclude
$$|\Ia f(X)| \le\frac{1}{\G(\alpha+1)} \Ma^\star f(X) \max\left\{\left(\frac{\bar{C}||f||_p}{\Ma^\star f(X)}\right)^{\frac{\alpha p}{D_0}},\left(\frac{\bar{C}||f||_p}{\Ma^\star f(X)}\right)^{\frac{\alpha p}{D_\infty}}\right\} + \frac{\bar{C}}{\G(\alpha)}||f||_p\, g(\ve_{min}).$$
In the case $0<A_{f}(X)<1$, then we have
$$\max\left\{\left(\frac{\bar{C}||f||_p}{\Ma^\star f(X)}\right)^{\alpha p/D_0},\left(\frac{\bar{C}||f||_p}{\Ma^\star f(X)}\right)^{\alpha p/D_\infty}\right\}=\left(\frac{\bar{C}||f||_p}{\Ma^\star f(X)}\right)^{\alpha p/D_0},\quad\mbox{ and }$$
\begin{align*}
g(\ve_{min})&=\frac{1}{\frac{D_\infty}{p}-\alpha}-\frac{1}{\frac{D_0}{p}-\alpha}+\frac{1}{\frac{D_0}{p}-\alpha}\left(\frac{\bar{C}||f||_p}{\Ma^\star f(X)}\right)^{-\frac{p}{D_0}(\frac{D_0}{p}-\alpha)}\\
&\leq \frac{1}{\frac{D_\infty}{p}-\alpha}\left(\frac{\bar{C}||f||_p}{\Ma^\star f(X)}\right)^{-\frac{p}{D_0}(\frac{D_0}{p}-\alpha)}.
\end{align*}
We conclude that, if $A_{f}(X)<1$, then
$$|\Ia f(X)| \le C_1 \Ma^\star f(X)^{1-\frac{\alpha p}{D_0}} ||f||_p^{\frac{\alpha p}D_0}.$$
If instead $A_{f}(X)\geq 1$, we can show in a similar way that
$$|\Ia f(X)| \le C_2 \Ma^\star f(X)^{1-\frac{\alpha p}{D_\infty}} ||f||_p^{\frac{\alpha p}D_\infty}.$$
If we write
$$\Ia f=\Ia f \cdot \chi_{\{A_{f}<1\}} +\Ia f\cdot\chi_{\{A_{f}\geq 1\}},$$
we deduce from (b) in Theorem \ref{T:maximal} that $\Ia f\chi_{\{A_{f}<1\}}\in L^{q_0}$ and $\Ia f\chi_{\{A_{f}\geq 1\}}\in L^{q_\infty}$ with the relative bounds
$$||\Ia f\chi_{\{A_{f}<1\}}||_{q_0}\leq c' ||f||_p,\qquad ||\Ia f\chi_{\{A_{f}\geq 1\}}||_{q_\infty}\leq c'' ||f||_p.$$
This proves \eqref{strongsum}.\\
Let us turn to the proof of (i). Let $p=1$, $0<\alpha<D_\infty<D_0$, and suppose $||f||_1\neq 0$. Recalling \eqref{Iageps} and the invertibility of $g$, for every positive $\lambda$ we can pick $\ve>0$ such that $\frac{\bar{C}}{\G(\alpha)}||f||_1 g(\ve)=\lambda$. From (a) in Theorem \ref{T:maximal}, we then get
\begin{align}\label{fromhereinvert}
|\{X\in \RN\mid |\Ia f(X)|> 2 \la\}| &\le \left|\left\{X\in \RN\mid \frac{1}{\G(\alpha+1)} \Ma^\star f(X)\ \ve^\alpha > \la\right\}\right|\\
 &\le \frac{2A \ve^\alpha}{\la \G(\alpha+1)}  ||f||_1.\nonumber
\end{align}
We can compute explicitly the inverse of $g$ and find an expression for $\ve$. In fact, if $y$ belongs to the interval $(0,(D_\infty-\alpha)^{-1})$ we have $g^{-1}(y)=\left((D_\infty-\alpha) y\right)^{\frac{1}{\alpha-D_\infty}}$, otherwise we have $g^{-1}(y)=\left(1-\frac{D_0-\alpha}{D_\infty-\alpha}+(D_0-\alpha) y\right)^{\frac{1}{\alpha-D_0}}$. Hence, if $\frac{\lambda \G(\alpha)}{\bar{C}||f||_1}<(D_\infty-\alpha)^{-1}$, we deduce from \eqref{fromhereinvert} that
$$|\{X\in \RN\mid |\Ia f(X)|> 2 \la\}| \le \frac{2A}{\la \G(\alpha+1)}  ||f||_1 \left( \frac{\lambda (D_\infty-\alpha)\G(\alpha)}{\bar{C}||f||_1}\right)^{\frac{\alpha}{\alpha-D_\infty}}=C_m \left(\frac{||f||_1}{\lambda}\right)^{\frac{D_\infty}{D_\infty-\alpha}}.$$
On the other hand, if $\frac{\lambda \G(\alpha)}{\bar{C}||f||_1}\geq(D_\infty-\alpha)^{-1}$, we have
\begin{align*}
&|\{X\in \RN\mid |\Ia f(X)|> 2 \la\}| \le \frac{2A}{\la \G(\alpha+1)}  ||f||_1 \left(1-\frac{D_0-\alpha}{D_\infty-\alpha}+\frac{\lambda (D_0-\alpha) \G(\alpha)}{\bar{C}||f||_1}\right)^{\frac{\alpha}{\alpha-D_0}}\\
&=\frac{2A \bar{C}^{\frac{\alpha}{D_0-\alpha}}}{\G(\alpha+1)\left((D_0-\alpha) \G(\alpha)\right)^{\frac{\alpha}{D_0-\alpha}}}\left(\frac{||f||_1}{\lambda}\right)^{\frac{D_0}{D_0-\alpha}}\left(1- \frac{\bar{C}||f||_1}{\lambda \G(\alpha)}\left(\frac{1}{D_\infty-\alpha}- \frac{1}{D_0-\alpha}\right)\right)^{-\frac{\alpha}{D_0-\alpha}}\\
&\leq \frac{2A \bar{C}^{\frac{\alpha}{D_0-\alpha}}}{\G(\alpha+1)\left((D_0-\alpha) \G(\alpha)\right)^{\frac{\alpha}{D_0-\alpha}}}\left(\frac{D_\infty-\alpha}{D_0-\alpha}\right)^{-\frac{\alpha}{D_0-\alpha}}\left(\frac{||f||_1}{\lambda}\right)^{\frac{D_0}{D_0-\alpha}}=C_M \left(\frac{||f||_1}{\lambda}\right)^{\frac{D_0}{D_0-\alpha}}.
\end{align*}
In any case, for any $\lambda>0$, we get
$$\min\left\{ \la\ |\{X\mid |\Ia f(X)|>2\la\}|^{\frac{D_0-\alpha}{D_0}}, \la\ |\{X\mid |\Ia f(X)|>2\la\}|^{\frac{D_\infty-\alpha}{D_\infty}}\right\} \le S ||f||_{L^1}, $$
where $S_1=S_1(D_0, D_\infty, \alpha, A, \bar{C})>0$. This proves \eqref{weakL1sum}.

\end{proof}

Using Theorem \ref{T:mainsum}, we obtain the following substitute result for Theorem \ref{T:sob}. We leave it to the interested reader to fill the necessary details.

\begin{theorem}\label{T:sobsum}
Suppose that \eqref{vol2} hold. Let $0<s< 1$. Given $1\leq p<D_\infty/2s<D_0/2s$, let $q_\infty>q_0>p$ be such that
$\frac 1p - \frac{1}{q_\infty} = \frac{2s}{D_\infty}$, $\frac 1p - \frac{1}{q_0} = \frac{2s}{D_0}$.
\begin{itemize}
\item[(a)] If $p>1$ we have $\Lo \ \hookrightarrow\ L^{\frac{pD_\infty}{D_\infty-2sp}}+L^{\frac{pD_0}{D_0-2sp}}.$ More precisely, there exists a constant $S_{p,s} >0$, depending on $N,D_\infty, D_0, s,\gamma, p$, such that for any $f\in \So$ one has
\[
||f||_{L^{q_0}+L^{q_\infty}} \le S_{p,s} ||\As f||_p.
\] 
\item[(b)]  If instead $p = 1$, we have $\mathscr L^{2s,1}\ \hookrightarrow\ L^{\frac{D_0}{D_0-2s},\infty} + L^{\frac{D_\infty}{D_\infty-2s},\infty}.
$ More precisely, there exists a constant $S_{1,s} >0$, depending on $N,D_\infty, D_0, s,\gamma$, such that for any $f\in \So$ one has
\[
\min\{\underset{\la>0}{\sup}\ \la\ |\{X\mid |f(X)|>\la\}|^{\frac{1}{q_0}}, \underset{\la>0}{\sup}\ \la\ |\{X\mid | f(X)|>\la\}|^{\frac{1}{q_\infty}} \} \le S_{1,s} ||\As f||_{1}.
\]
\end{itemize}
\end{theorem}




\bibliographystyle{amsplain}

\end{document}